\documentclass[11pt]{amsart}
\usepackage{extarrows}
\usepackage{latexsym}

\usepackage{multirow}
\usepackage{makecell, colortbl}
\usepackage{booktabs, makecell, tabularx}
\usepackage{lscape}

\usepackage{wrapfig}
\usepackage{multicol}
\usepackage{mathdots}
\usepackage{stmaryrd}

\usepackage{amssymb,amsmath,amscd}
\usepackage[pdftex]{graphicx}
\usepackage{enumerate}
\usepackage[dvipsnames]{xcolor}
\usepackage{tikz, ifthen}
\usetikzlibrary{patterns}
\usepackage[lmargin=1in, rmargin=1in, tmargin=0.95in, bmargin=0.95in]{geometry}
\usepackage{listings}
\usepackage{courier}
\usepackage{tikz-cd, wrapfig}
\usetikzlibrary{math}
\usetikzlibrary{decorations.pathreplacing}
\usepackage{color}
\usepackage{MnSymbol}
\usepackage{hyperref}
\usepackage{mathrsfs, colonequals}

\newcommand{\pp}{\mathbb{P}}
\newcommand{\qq}{\mathbb{Q}}

\newcommand{\zz}{\mathbb{Z}}

\newcommand{\C}{\mathcal{C}}

\renewcommand{\H}{\mathcal{H}}

\renewcommand{\L}{\mathcal{L}}
\newcommand{\M}{\mathcal{M}}
\renewcommand{\O}{\mathcal{O}}

\newcommand{\R}{\mathcal{R}}

\newcommand{\N}{\mathcal{N}}
\newcommand{\D}{\mathcal{D}}

\newcommand{\W}{\mathcal{W}}

\newcommand{\GL}{\text{GL}}

\newcommand{\PGL}{\operatorname{PGL}}
\newcommand{\Sym}{\operatorname{Sym}}

\newcommand{\Spec}{\operatorname{Spec}}

\newcommand{\Pic}{\operatorname{Pic}}

\renewcommand{\Im}{\operatorname{Im}}

\newcommand{\End}{\operatorname{End}}

\newcommand{\ev}{\mathrm{ev}}

\renewcommand{\bar}{\overline}
\newcommand{\Gr}{\operatorname{Gr}}

\newcommand{\BN}{\operatorname{BN}}

\makeatletter
\newcommand{\oset}[3][0ex]{%
  \mathrel{\mathop{#3}\limits^{
    \vbox to#1{\kern-2\ex@
    \hbox{$\scriptstyle#2$}\vss}}}}
\makeatother

\newcommand{\posmod}{\oset{+}{\rightarrow}}

\newcommand{\negmod}{\oset{-}{\rightarrow}}

\newcommand{\defi}[1]{\textsf{#1}} 
\usetikzlibrary{decorations.pathreplacing}

\newtheorem{thm}{Theorem}[section]

\newtheorem{lem}[thm]{Lemma}
\newtheorem{conj}[thm]{Conjecture}
\newtheorem{prop}[thm]{Proposition}

\newtheorem{cor}[thm]{Corollary}

\theoremstyle{definition}
\newtheorem{defin}[thm]{Definition}

\newtheorem{example}[thm]{Example}

\newtheorem*{question*}{Question}
\newtheorem{question}[thm]{Question}
\newtheorem{warning}[thm]{{\fontencoding{U}\fontfamily{futs}\selectfont\char 49\relax} Warning}

\theoremstyle{remark}
\newtheorem{rem}{Remark}

\DeclareMathVersion{normal2}

\title{Recent advances in Brill--Noether theory and the geometry of Brill--Noether curves}
\author{Isabel Vogt}

\newlength{\myindent}
\begin{document}

\maketitle

\begin{abstract}
The first goal of this article is to survey recent progress in Brill--Noether theory, including both the 
the study of the moduli space of maps from a curve to projective space and the geometry of the resulting curves in projective space.  The second goal is to introduce newcomers to some of the important techniques that have been introduced or developed in the last decade that made these advances possible.
\end{abstract}

\section{Introduction}

Brill--Noether theory is one of the oldest subjects in algebraic geometry.  Classically, an algebraic curve was synonymous with a \(1\)-dimensional subset of projective space defined by polynomial equations.  Riemann's definition of an abstract curve split this study in two: the study of abstract curves of genus \(g\) and their moduli space \(\M_g\), and the study of all ways of mapping an abstract curve of genus \(g\) to projective space \(\pp^r\) with degree \(d\).  This second subject is Brill--Noether theory.  The name is in honor of Brill and M.~Noether who conjectured in 1874 a precise formula for when a \textit{general} curve of some genus admits a projective realization with specified invariants \cite{bn}.  It took until 1980 for this result to be proven by Griffiths--Harris \cite{gh_80}, building on earlier work of Kempf \cite{kempf} and Kleimann--Laksov \cite{kl}.  Subsequent work in the 1980s by Gieseker \cite{gp}, Fulton--Lazarsfeld \cite{fl} and Eisenbud--Harris \cite{im} culminated in the proof of the Brill--Noether Theorem (see Theorem~\ref{thm:BN} below), which gives a good description of the moduli space of maps from a general curve of genus \(g\) to \(\pp^r\) of degree \(d\).
Since then, there have been many new and influential proofs of aspects of this result using limit linear series \cite{EH_gp}, K3 surfaces \cite{laz_bn}, Bridgeland stability \cite{bayer}, and tropical geometry \cite{cdpr, jp_gp}.

Despite all of this work at the end of the 20th century, Brill--Noether theory remains an active area today.  The progress is focused in three main areas: (1) Brill--Noether theory for special curves and understanding the loci in \(\M_g\) consisting of curves with unexpected projective realizations, (2) the geometry of the genus \(g\) and degree \(d\) curves in \(\pp^r\) whose existence is guaranteed by the Brill--Noether Theorem, and (3) applications of Brill--Noether theory to the birational geometry of \(\bar{\M}_{g}\).
In the first area, the last decade has seen a complete analogue of the Brill--Noether theorem for curves of fixed gonality \cite{l_rbn, cpj1, cpj2, llv}, as well as progress on dimension and (ir)reducibility of the loci in \(\M_g\) with unexpected \(g^r_d\)'s \cite{Pf_neg}, and a complete answer to which loci are maximal with respect to containment \cite{AHK}.  Progress in the second area includes a proof of Severi's Maximal Rank Conjecture from 1915 \cite{mrc} and a complete solution to the interpolation problem for Brill--Noether curves \cite{interpolation}.  Finally, progress on the Strong Maximal Rank Conjecture has yielded a proof that \(\bar{\M}_{22}\) and \(\bar{\M}_{23}\) are of general type \cite{fjp}.

This article is organized as follows.  In Section~\ref{sec:background} we introduce the main objects of study and the classic Brill--Noether Theorem, with an eye towards developments in the past decade.  Section~\ref{sec:results} is devoted to a survey of these recent advances, covering both results on moduli spaces and new insights into the geometry of the curves in projective space.  Section~\ref{sec:techniques} has a different emphasis.  It is intended for readers new to the subject who wish to gain familiarity with some of the techniques underlying these recent breakthroughs.  To keep the exposition accessible, rather than sketching the most cutting edge results, we present detailed modern treatments of simpler results whose original proofs appeared earlier.  
Some of these arguments do not appear elsewhere in the literature, including the streamlined proof of the Brill--Noether Nonexistence Theorem in Section~\ref{sec:BN_nonexist}, and a degenerative proof of the Brill--Noether Existence Theorem in Section~\ref{sec:BN_exist}, although they are known to experts.  We hope their inclusion here will make these techniques more boradly accessible.

The topics touched on in this article are necessarily skewed towards the topics closest to my own work in the subject.  It should not be viewed as an exhaustive account of all interesting and important work in Brill--Noether theory in the last decade.  For a survey with both distinct and overlapping topics, see \cite{jp_survey}.  On a related note, the focus is exclusively on maps from a curve to projective space, and so completely ignores Brill--Noether theory for higher rank vector bundles on a curve and Brill--Noether questions for higher dimensional varieties.  We refer the reader to the recent surveys \cite{gm_survey, newstead_survey,  cn_survey} for introductions to these topics.

\subsection*{Acknowledgements} Thank you to Izzet Coskun, Gavril Farkas, Richard Haburcak, Joe Harris, Dave Jensen, Eric Larson, Hannah Larson, Feiyang Lin, and Nathan Pflueger for conversations that have shaped my understanding of Brill--Noether theory and for feedback on this article.  I also thank the National Science Foundation for supporting my work through DMS-2338345.

\section{Background}\label{sec:background}

Let \(C\) be a smooth, projective curve defined over an algebraically closed field \(K\).  The motivating question of Brill--Noether theory is the following.

\begin{question}\label{ques:BN}
What is the geometry of the space of maps from \(C\) to \(\pp^r\) of degree \(d\)?
\end{question}

A map \(C \to \pp^r\) is called \defi{nondegenerate} if the image is not contained in a hyperplane.  It makes sense to study nondegenerate maps, since a degenerate map to \(\pp^r\) is a nondegenerate map to a smaller projective space.
Recall that the data of a nondegenerate map \(C \to \pp^r\) of degree \(d\) is equivalent to the data of a line bundle \(L\) on \(C\) of degree \(d\) and \(r+1\) sections \(\sigma_0, \dots, \sigma_r \in H^0(C, L)\) spanning a basepoint-free subspace of dimension \(r+1\).  A fundamental player in this story is therefore the locus
\begin{equation}\label{bn_loci}
W^r_dC \colonequals \{[L] \in \Pic^d_{C} : \dim H^0(C, L) \geq r+1\}.\end{equation}
(When \(r = -1\), we recover \(W^{-1}_dC = \Pic^d_{C}\).)  The locus \(W^r_dC\) inherits a scheme structure as a degeneracy locus of a map of vector bundles, as we now explain; see \cite[Chapter IV, Section 3]{acgh} for more details.  Let \(\L\) be a Poincar\'e bundle on \(C \times \Pic^d_C\).  Consider the second projection \(\pi \colon C \times \Pic^d_C \to \Pic^d_C\). Given points \(p_1, \dots, p_N \in C\) for \(N \gg 0\), let \(D \colonequals p_1 + \cdots + p_N\) and let \(\D \colonequals p_1 \times \Pic^d_C + \cdots + p_N \times \Pic^d_C\) be the corresponding divisor of relative degree \(N\).  By the theorem on cohomology and base change, \(\pi_*\L(\D)\) is a vector bundle of rank \(d + N + 1 - g\) whose fiber over a point \([L] \in \Pic^d_C\) is isomorphic to \(H^0(C, L(D))\).  By construction, the pushforward \(\pi_* \L(\D)|_{\D}\) is a vector bundle of rank \(N\) whose fiber over a point \([L]\in \Pic^d_C\) is isomorphic to \(L(D)|_D\).  Restriction to \(\D\) gives a map
\begin{equation}\label{eq:bn_degen}
 \pi_* \L(\D) \to \pi_*\L(\D)|_{\D},
\end{equation}
specializing to \(H^0(C, L(D)) \to L(D)|_D\) over \([L]\).
The locus \(W^r_dC\) is defined to be the degeneracy locus where the map \eqref{eq:bn_degen} has a kernel of dimension at least \(r+1\), i.e., where it has rank at most \(d -g -r + N\).  This description is independent of the choices made \cite[Chapter IV, Remark 3.2]{acgh}.

The first natural question about \(W^r_dC\) is when it is nonempty.  Of course, this is a subtle question that depends on the geometry of the curve \(C\).  
For example, if \(C\) is a hyperelliptic curve of genus \(g\) and \(0< n< g-1\), then the \(n\)th power of the hyperelliptic line bundle exhibits that \(W^n_{2n} C\) is nonempty; but Clifford's theorem guarantees that \(W^{n}_{2n} C = \emptyset\) if \(C\) is nonhyperelliptic \cite[Chapter III, Section 1]{acgh}.
Nevertheless, the fact that \(W^r_dC\) is a degeneracy locus of the map \eqref{eq:bn_degen} of vector bundles means that it has expected codimension 
\[\text{(dim of kernel)\(\cdot\)(dim of cokernel)} = (r+1)(g-d+r)\]
in the \(g\)-dimensional \(\Pic^d_C\).  This is the modern phrasing of the calculation that Brill and Noether did to conjecture that \textit{for general \(C\)} the locus \(W^r_dC\) is nonempty if and only if
\begin{equation}\label{eq:BN}\tag{BN}
\rho(g, r, d) \colonequals g - (r+1)(g-d+r) \geq 0.
\end{equation}
The quantity \(\rho(g, r, d)\) is known as the \defi{Brill--Noether number}.  

\begin{thm}[The Brill--Noether theorem]\label{thm:BN}
Fix integers \(g, r, d\) with \(g \geq 2\).
Let \(C\) be a general curve of genus \(g\) (i.e., \([C]\) is in a dense open in \(\M_g\)).
\begin{enumerate}
\item (Griffiths--Harris \cite[Main Theorem II(a)]{gh_80}) \(W^r_d C\) has dimension \(\min(g, \rho(g, r, d))\). In particular, if \(\rho(g, r, d) < 0\), then \(W^r_dC = \emptyset\). \label{eq:BN_dim}
\item (Gieseker \cite[Theorem 1.1]{gp}, cf. Remark~\ref{rem:gp}) The locus \(W^r_d C\) is normal, Cohen--Macaulay,  and is smooth away from \(W^{r+1}_d C\).  In characteristic \(0\), it has rational singularities. \label{eq:BN_gp}
\item (Kempf \cite[Corollary, page 15]{kempf}, Kleiman--Laksov \cite{kl}, \cite[Proposition 4]{kl2}) When \(\rho(g, r, d) = 0\), 
\begin{align*}
N(g, r, d)   \colonequals \# W^r_d C &= g!\prod_{\alpha = 0}^r \frac{\alpha!}{(g - d + r + \alpha)!}\\
&=\text{ \# Standard Young Tableaux on \((r+1) \times (g - d+r)\) diagram.} 
\end{align*}
\item (Fulton--Lazarsfeld \cite[Corollary, page 1]{fl}, Eisenbud--Harris \cite[Theorem 1]{im}, Edidin \cite[Theorem 1]{edidin}) When \(\rho(g,r,d) > 0\),  \(W^r_d C\) is irreducible.  
When \(\rho(g,r,d) = 0\), the universal Brill--Noether locus \(\mathcal{W}^r_d\) has a unique component dominating \(\M_g\), and if \(r+1 \neq g- d + r\), then the monodromy of this cover contains \(A_{N(g, r, d)}\). \label{eq:BN_mon}
\end{enumerate}
\end{thm}

Some aspects of the Brill--Noether Theorem hold for arbitrary smooth curves \(C\) of genus \(g\): if \(\rho(g, r, d) \geq 0\), then the locus \(W^r_dC\) has dimension at least \(\min(g, \rho(g, r, d))\) \cite[Section 4]{kempf}, \cite{kl} (see also \cite[Main Theorem I]{gh_80}) and  if \(\rho(g, r, d) > 0\), then \(W^r_dC\) is connected \cite[Theorem 1]{fl}.

The original proofs of these results often assumed that the ground field had characteristic \(0\), but this assumption can be removed from all parts.  
It was observed by Welters \cite[Section 2.7]{welters} that that in the context of limit linear series, chains of elliptic curves are better-suited to studying Brill--Noether theory in positive characteristic.  Imitating the limit linear series proofs in the literature \cite{EH_gp, im}, while replacing flag curves with chains of elliptic curves, can yield characteristic-independent proofs of all parts of Theorem~\ref{thm:BN}.  Implementing this replacement is most subtle in the \(\rho(g,r,d)=0\) case of part \eqref{eq:BN_mon}, see~\cite[Section 10]{llv}.
Theorem 1 and the Corollary in \cite{fl} hold in arbitrary characteristic \cite[Remark 2.8]{fl}.  

While the original proofs of parts~\eqref{eq:BN_dim} and \eqref{eq:BN_gp} of Theorem~\ref{thm:BN} used degeneration, Lazarsfeld reproved this in characteristic \(0\) using K3 surfaces and avoiding degeneration \cite[Theorem, page 1]{laz_bn}.  
The impact of the work \cite{laz_bn} is hard to overstate: the new tools it introduced to the study of Brill--Noether general curves made possible, for example,  Voisin's proof of the generic case of Green's Conjecture \cite{v_green1, v_green2}, as well as more recent breakthroughs, like the determination of the maximal Brill--Noether loci \cite{AHK} presented in Theorem~\ref{thm:max} below.
Recently, building off of this perspective, Arbarello--Bruno--Farkas--Sacc\`a \cite[Corollary 1.3 and Section 5]{abfs} gave explicit constructions of smooth Brill--Noether general curves of every genus defined over \(\qq\).

The Brill--Noether theorem is most interesting when \(d < g+ r\), and so by the Riemann--Roch theorem, a bundle \(L\) in \(W^r_dC\) necessarily has \(h^1(C, L) \neq 0\).  In the regime, Theorem~\ref{thm:BN}\eqref{eq:BN_dim} implies that a general line bundle \(L\) in \(W^r_dC\) has 
\(h^0(C, L)  = r+1\).  

\begin{rem}[The Gieseker--Petri Theorem and the singularities of \(W^r_dC\)]\label{rem:gp}
The \defi{Petri map} for a line bundle \(L\) is the multiplication map
\begin{equation}
\mu_L \colon H^0(C, L) \otimes H^0(C, K_C\otimes L^\vee) \to H^0(C, K_C).
\end{equation}
Gieseker proved in \cite[Theorem 1.1]{gp} that if \(C\) is general, then for every \(L \in \Pic^d_C\), the Petri map \(\mu_L\) is injective.  Since \(W^r_dC\) is a degeneracy locus of \eqref{eq:bn_degen}, choosing frames for the source and target yields a natural \(  \GL_{d + N +1-g}\times \GL_{N} \)-bundle over \(W^r_dC\).  The injectivity of the Petri map then guarantees that this bundle is smooth over the corresponding universal determinatal locus in the space of matrices.  This universal determinantal locus:
\begin{itemize}
\item Is smooth away from the determinantal sublocus corresponding to strictly smaller rank.
\item Is normal and Cohen--Macaulay \cite[Propositions 6.1.5(a)-(b) and 6.2.3(b)]{weyman}.   
\item Has rational singularities in characteristic \(0\) \cite[Proposition 6.15(b)]{weyman}. 
\end{itemize}
Hence \(W^r_dC\) shares all of these properties.
A resolution of \(W^r_dC\) is given by the locus
\[G^r_dC = \{(L, V) : L\in W^r_d C, V \in \Gr(r+1,  H^0(C, L))\},\]
which is naturally defined from the degeneracy locus formalism, as we now recall.  Consider the Grassmann bundle \(\psi \colon \Gr(r+1, \pi_* \L(\D))\to \Pic^d_C\) with universal subbundle \(S \to \psi^* \pi_* \L(\D)\).  The locus \(G^r_dC\) is defined to be the kernel of the composite map \(S \to \psi^* \pi_*\L(\D)|_{\D}\)
 (see \cite[Chapter IV, Section 3]{acgh} for more details).
\end{rem}

Another immediate consequence of Theorem~\ref{thm:BN}\eqref{eq:BN_dim} in this regime is that when \(r \geq 1\), a general line bundle in \(W^r_dC\) is basepoint free; indeed, a line bundle with a basepoint necessarily lies in the image of \(W^r_{d-1}C \times (W^0_1C \simeq C)\) under the tensor product map \(\Pic^{d-1}_C \times \Pic^1_C \to \Pic^d_C\), which has dimension at most \(\rho(g, r, d-1) + 1 < \rho(g, r, d)\).
When \(d \geq g + r\), using Serre duality, the locus of line bundles with a basepoint lies in the image of \(W^0_{2g - d -1} \times W^0_1 \to W^r_d\) sending \((M, p)\) to \(K_C \otimes M(p)\), which again has dimension at most \(2g - d \leq g-r \leq g-1\).
The Brill--Noether theorem, therefore, gives a good answer to Question~\ref{ques:BN} when \(C\) is of general moduli.

Although this is phrased for a single curve \(C\), the fact that it holds over a dense open in \(\M_g\) means that the Brill--Noether Theorem has the following consequence for the moduli space \(\M_g(\pp^r, d)\) of maps of degree \(d\) from a smooth curve of genus \(g\) to \(\pp^r\).

\begin{cor}\label{cor:BN_comp}\hfill
\begin{enumerate}
\item If \(\rho(g, r, d) \geq 0\),  there is a unique irreducible component \(\M_g(\pp^r, d)^{\BN}\) of \(\M_g(\pp^r, d)\) parameterizing nondegenerate maps and for which the map \(\M_G(\pp^r, d)^{\BN} \to \M_g\) taking the moduli of the source curve is dominant.  If \(\rho(g, r, d) < 0\),  no such component exists.
\item The \defi{Brill--Noether component} \(\M_g(\pp^r, d)^{\BN}\) is generically smooth of dimension 
\[\dim \M_g + \dim \PGL_{r+1} + \rho(g, r, d) = (r+1)d - (r-3)(g-1).\]
\end{enumerate}
\end{cor}

The good description of the moduli space \(\M_g(\pp^r, d)\) afforded by the Brill--Noether theorem raises the question of understanding the geometry of the resulting maps.  A first step in this direction is the following.

\begin{thm}[{The Embedding Theorem, Eisenbud--Harris \cite[Theorem 1]{EH83} for \(p=1\), Farkas \cite[Corollary 0.3]{farkas_p}}]\label{thm:emb}
If \(r \geq 2p + 1\), then a general map in \(\M_g(\pp^r, d)^{\BN}\) is \(p\)-very ample.
\end{thm}

In particular, the general map in the Brill--Noether component is an embedding once \(r \geq 3\).  Justified by this, 
we term the maps parameterized by the Brill--Noether component \defi{Brill--Noether curves}.  This gives a precise meaning to the term ``general curve in projective space."

Another consequence of the fact that the Brill--Noether theorem holds for a general curve is that the complement is an interesting naturally defined locus in \(\M_g\).

\begin{defin}
The \defi{Brill--Noether locus} \(\M^r_{g, d}\) is the locus of curves \(C\in \M_g\) for which \(W^r_dC \neq \emptyset\). 
\end{defin}

When \(r=0\) and \(d > 0\), \(W^0_dC = \Im(\Sym^d_C \to \Pic^d_C)\) is always nonempty and so \(\M^0_{g, d} = \M_g\).
By Serre duality we have \(\M^r_{g,d}= \M^{g-d+r-1}_{g, 2g-2-d}\); to remove this redundancy we may assume that \(d \leq g-1\).

When \(\rho(g, r, d) \geq 0\), the locus \(\M^r_{g,d}\) is all of \(\M_g\), but when \(\rho(g, r, d) < 0\) it is a proper subvariety by the Brill--Noether Theorem.
The expected codimension of \(\M^r_{g, d}\) is \(\max(0,-\rho(g, r, d))\).  Steffen proved in \cite[Theorem 0.1]{steffen} that every component of \(\M^r_{g, d}\) has codimension at most \(\max(0,-\rho(g, r, d))\). 
If \(\rho(g, r, d) = -1\), this combined with a theorem of Eisenbud--Harris \cite[Theorem (1.1)(ii)]{eh_m1} shows that \(\M^r_{g,d}\) is an irreducible divisor in \(\M_g\).  Eisenbud--Harris computed the class of the closure of this divisor (up to an overall positive constant) in \(\bar{\M}_g\) \cite[Theorem 1]{eisenbud_harris_mg}.
When \(g+1\) is composite, this suffices to prove that \(\M_g\) is of general type for \(g \geq 24\) \cite[Theorem A]{eisenbud_harris_mg}.  When \(\rho(g, r, d) \leq -2\),  \cite[Theorem (1.1)(i)]{eh_m1} shows that \(\M^r_{g,d}\) has no divisorial components, and so, combined with  \cite[Theorem 0.1]{steffen}, this proves that when \(\rho(g, r, d) = -2\), the locus \(\M^r_{g,d}\) is equidimensional of codimension \(2\).  
Building off of work of Edidin \cite{edidin_rhom2} that determines a basis for the codimension \(2\) rational homology of \(\bar{\M}_{g}\) when \(g \geq 12\), Tarasca computed the classes of the closures of the codimension \(2\) loci \(\M^1_{2k, k}\) in the tautological ring of \(\bar{\M}_{2k}\) \cite[Theorem 2]{tarasca}.
Equidimensionality of \(\M^r_{g,d}\) when \(\rho(g, r, d) = -3\) and \(g \geq 12\) follows from work of Edidin \cite[Theorem 0.1]{edidin_m3} (again, utilizing \cite{edidin_rhom2}), combined with \cite{eh_m1, steffen}.

\section{Important recent results}\label{sec:results}

\subsection{Moduli aspects of Brill--Noether theory}
Since the Brill--Noether theorem gives a good description of \(W^r_dC\) for a general curve \(C\), most recent work focuses on the setting where the Brill--Noether theorem fails: namely for \textit{special} curves.  Let us begin with what is known about the Brill--Noether loci \(\M^r_{g,d}\) themselves when \(\rho(g, r, d) < 0\).  
When \(\rho(g, r, d)\) is quite negative (on the order of  \(-\dim \M_g\) or smaller), it is not hard to find components of \(\M^r_{g,d}\) with dimension strictly bigger than the expected dimension \(\min(-1, 3g-3+\rho(g, r, d))\): for example Castelnuovo extremal curves curves in \(\pp^r\) for \(r\geq 3\) and \(d\) sufficiently large give rise to a component of unexpectedly large dimension in \(\M_g\) \cite[Remark 1.4]{Pf_neg}.
The results of Hurwitz--Brill--Noether theory summarized in Section~\ref{sec:hbn} can also be used to give components of unexpectedly large dimension as in \cite[Appendix A.1]{Pf_neg}; we give an explicit such example from \cite{ht_red} in Example~\ref{ex:1123} in the next subsection.

On the one hand, there are some values of \((r,d)\) for which the irreducibility and codimension of \(\M^r_{g,d}\) is well-understood.  For example,
\(\M^1_{g, k}\) is the image of the Hurwitz space \(\H_{k,g}\) of degree \(k\) genus \(g\) covers of \(\pp^1\) (up to projective equivalence on \(\pp^1\)) since the \(g^1_k\) on a general curve in \(\M^1_{g,k}\) is basepoint free\footnote{Indeed, it suffices to show that the locus of curves in \(\M_g\) with a \(g^1_k\) with a basepoint is in the closure of the image of \(\H_{k,g}\).  This can be proven by smoothing out the nodal union of a degree \(k-m\) cover and \(m\) copies of \(\pp^1\) mapping with degree \(1\).}.  It is classical that the dimension of \(\H_{k,g}\) is \(2g + 2k - 5 = 3g - 3 + \rho(g, 1, k)\).  Furthermore, it is now known by recent work of Christ--He--Tyomkin \cite[Theorem A]{cht2} that \(\H_{k,g}\) is irreducible in all characteristics (in characteristics greater than \(k\) this is due to Fulton \cite{fulton_irred}).   
Using deformation theory, one can show that the map \(\H_{k,g} \to \M_g\) is generically finite onto its image when \(\rho(g, 1, k) < 0\) and so \(\M^1_{g,k}\) is always irreducible of the expected dimension.

In another direction, instead of fixing \(r\) to be small, we can fix \(-\rho(g, r, d)\) to be small.  The most systematic result in this direction is the following.

\begin{thm}[{Pflueger \cite[Theorem A]{Pf_neg}}]\label{thm:Pf_neg}
Suppose that \(r+1\) and \(g -d + r\) are both at least \(2\) and that \(-g+3\leq \rho(g, r, d) < 0\).  Then \(\M^r_{g,d}\) has an irreducible component of the expected dimension.
\end{thm}

This result was subsequently independently proved by Teixidor i Bigas in the slightly smaller range where \(-g +(r+1) \leq \rho(g,r,d) <0\) \cite[Theorem 2.1]{teix25}.  In \cite[Theorem B]{Pf_neg}, Pflueger proves a similar result subject to the asymptotic \(-(3 + c(r))g + O(g^{5/6}) \leq \rho(g, r, d) < 0\), where \(c(r)\) is an explicit constant going to \(0\) as \(r \to \infty\).  In the regime where \(r\) is large, this result is substantially stronger than Theorem~\ref{thm:Pf_neg}, considering that the smallest value of \(\rho(g, r, d)\) where such a theorem could hold is \(-3g+3\).

We now turn to the question of understanding the Brill--Noether Question~\ref{ques:BN} for curves \(C \in \M^r_{g,d}\) when \(\rho(g, r, d) < 0\) and so the classic Brill--Noether Theorem~\ref{thm:BN} can fail.

\begin{question}\label{ques:spec}
Consider a curve \(C\) that is general in some component of \(\M^{r_0}_{g,d_0}\).  What can we say about \(W^{r}_{d}C\) for \((r, d)\neq (r_0, d_0)\)?
\end{question}

The first thing we could hope to say about \(W^{r}_{d}C\) is when it is (non)empty.  We will say that a \(g^{r}_{d}\) on a curve \(C\) is \defi{\(\rho\)-unexpected} if \(\rho(g, r, d) <0\) and \defi{\(\rho\)-expected} otherwise.  The Brill--Noether Theorem~\ref{thm:BN}\eqref{eq:BN_dim} guarantees that when \(C\) is general in \(\M_g\), there are no \(\rho\)-unexpected \(g^{r}_{d}\)'s on \(C\).
This leads to the following natural subquestion: 

\begin{question}\label{ques:max}
For which \((r_0, d_0)\) does a general curve in some component of \(\M^{r_0}_{g,d_0}\) have no \(\rho\)-unexpected \(g^{r}_{d}\) for \((r, d)\not\in \{ (r_0, d_0), (g - d_0 + r_0 -1, 2g-2-d_0)\}\)?  
\end{question}

If \((r_0, d_0)\) appears in the answer to Question~\ref{ques:max}, then \(\M^{r_0}_{g,d_0}\not\subset \M^{r}_{g, d}\) for every \((r, d)\not\in \{ (r_0, d_0), (g - d_0 + r_0 -1, 2g-2-d_0)\}\) for which \(\rho(g, r, d) < 0\).  The converse is not necessarily true, since it could be that curves from different components of \(\M^{r_0}_{g,d_0}\) have disjoint \(\rho\)-unexpected linear series.  Nevertheless, the \((r_0, d_0)\) appearing in the answer to Question~\ref{ques:max} give rise to (proper) Brill--Noether loci that are maximal with respect to containment.
There are some trivial containments of Brill--Noether loci:
\begin{flalign}
\label{add} \text{(Adding a point)} \ & \text{\(\M^r_{g, d} \subset \M^r_{g, d+1}\),}& \\
\label{subtract} \text{(Subtracting a general point)} \ & \text{\(\M^r_{g, d} \subset \M^{r-1}_{g, d-1}\),}&
\end{flalign}
which give rise to trivial instances of \(\rho\)-unexpected linear series whenever \(\rho(g, r_0, d_0+1) <0\) and \(\rho(g, r_0-1, d_0-1)<0\) respectively.

A Brill--Noether locus is called \defi{expected maximal} if it is maximal with respect to these trivial containments, equivalently if adding a point \eqref{add} and subtracting a general point \eqref{subtract} do not give rise, trivially, to \(\rho\)-unexpected \(g^{r}_{d}\)'s. 
Explicitly, \(\M^{r_0}_{g,d_0}\) is expected maximal if \(\rho(g, r_0, d_0) < 0\) and \(\rho(g, r_0, d_0+1) \geq 0\) and \(\rho(g, r_0-1, d_0-1) \geq 0\). 
From the fact that \(\rho(g, r_0, d_0) <0\) but \(\rho(g, r_0, d_0+1) \geq 0\) it follows that \(d_0 = \lceil(r_0g)/(r_0+1)\rceil +r_0 -1\) and \(\rho(g, r_0, d_0) \geq -r_0-1\), which is at least \(-\sqrt{g}\) by \cite[Lemma 1.1]{AHL}.  It is an elementary calculation that there also exists an expected maximal Brill--Noether locus with \(-\rho(g, r_0, d_0)\) on the order of \(g^{1/2 - \epsilon}\).  These expected maximal Brill--Noether loci therefore occur in moderately large expected codimension, where we know very little about the geometry of \(\M^{r_0}_{g,d_0}\), other than that there is \textit{some} component occurring in the expected codimension by Theorem~\ref{thm:Pf_neg}.

If \(\M^{r_0}_{g,d_0}\) is not expected maximal, then it is not maximal and \((r_0, d_0)\) does not appear in the answer to Question~\ref{ques:max}.
Inspired by work of Farkas \cite{farkas_g23} and Lelli-Chiesa \cite{lc_lm}, Auel and Haburcak conjectured in \cite[Conjecture 1]{AH}  that expected maximal Brill--Noether loci are actually maximal with respect to containment when \(g \geq 3\) and  \(g\neq 7, 8, 9\) (where counterexamples \cite[Section 6.2]{AH} were known; we explain one of them in Example~\ref{ex:814} using Hurwitz--Brill--Noether theory).  Subsequent partial progress by several authors \cite{AHL, B, BH, teix25} culminated in the following complete result, which both proves the Maximal Brill--Noether loci conjecture and answers the a priori\footnote{A posteriori, the \((r_0, d_0)\) appearing in the answer to Question~\ref{ques:max} correspond exactly to the maximal \(\M^{r_0}_{g, d_0}\).} stronger Question~\ref{ques:max}.

\begin{thm}[{Auel--Haburcak--Knutsen \cite[Theorem 1]{AHK}}]\label{thm:max}
 Let \(g \geq 3\) and \(2 \leq d \leq g-1\).  There exists a component of \(\M^{r_0}_{g,d_0}\) for which the general curve \(C\) in that component has no \(\rho\)-unexpected linear series \(g^{r}_{d}\) for \((r, d)\not\in \{ (r_0, d_0), (g - d_0 + r_0 -1, 2g-2-d_0)\}\) if and only if \((g, r_0, d_0)\) is expected maximal and \((g, r_0, d_0) \nin \{(7, 2, 6),(8, 1, 4),(9, 2, 7)\}\).
\end{thm}

This theorem guarantees that the answer to Question~\ref{ques:max} comes purely from ruling out trivially \(\rho\)-unexpected \(g^{r}_{d}\) from \eqref{add} and \eqref{subtract} once \(g \geq 10\).  This doesn't mean that the trivially \(\rho\)-unexpected \(g^{r}_{d}\)'s are the \textit{only} \(\rho\)-unexpected \(g^{r}_{d}\) in general!
Haburcak has completely classified all containments between (proper) Brill--Noether loci in genera up to \(12\) \cite[Theorem A]{H_g12}, and there are many subtle and interesting phenomena.

Theorem~\ref{thm:max} completely answer Question~\ref{ques:max}.  
We turn now to what is known more generally about Question~\ref{ques:spec}.  The most complete results are known in the case \(r_0=1\), which received substantial work in the last decade. Since \(\M^1_{g,k}\) is the image of the Hurwitz space \(\H_{k,g}\), this subject goes by the name of \defi{Hurwitz--Brill--Noether theory}; we survey the main results in the next Section~\ref{sec:hbn}.  Some scattered other results are known: \cite{LV_hir} studies the case \(r_0=2\) and \(g = {d_0-1 \choose 2}\) of smooth plane curves (and smooth curves on Hirzebruch surfaces more generally).  Their \cite[Theorem 1.1]{LV_hir} computes the dimension of \(W^r_dC\) for a smooth plane curve \(C\), and gives some smoothness results in characteristic \(0\), but does not address other aspects of the geometry.  The gonality of a plane curve of degree \(d\) with not too many nodes and cusps is computed to be \(d-2\) in \cite[Theorem 2.1]{CK_gon}.

\subsection{Hurwitz--Brill--Noether Theory}\label{sec:hbn}

In this section we study Question~\ref{ques:spec} when \(r_0=1\).  That is, let \(C\) be a genus \(g\) curve equipped with a map \(f\colon C \to \pp^1\) of degree \(k\).  Many authors observed that in this setting \(W^r_dC\) can be reducible of the ``wrong" dimension with unexpected singularities \cite{B89,CKM, BK, M96, CM99, CM00, CM02,  park}.  The first interesting example is the following.

\begin{example}[Trigonal genus \(5\) curve]\label{ex:trigonal_baby}
Let \(f \colon C \to \pp^1\) be a general degree \(3\) genus \(5\) cover and let \(H \colonequals f^*\O_{\pp^1}(1)\).  Then \(W^1_4C\) has two irreducible components \(X_1\) and \(X_2\), both isomorphic to \(C\):
\begin{center}
\begin{tikzpicture}
\draw (0,0) .. controls (1,1.5) and (3,1.5) .. (4,0);
\draw (0,1) .. controls (1,-.5) and (3,-.5) .. (4,1); 
\draw (0,0) node[left] {\(X_1 = \{H(p) : p \in C\}\)};
\draw (4,1) node[right] {\(\{K_C\otimes H^{-1}(-q) : q \in C\} = X_2 \)};
\filldraw (.43,.5) circle[radius=0.03];
\filldraw (4-.43,.5) circle[radius=0.03];
\draw (.43,.5) node[left]{\tiny \(H(p_1)\)};
\draw (4-.43,.5) node[right]{\tiny \(H(p_2)\)};
\end{tikzpicture}
\end{center}
where \(p_1 + p_2\) is the unique effective representative of \(K_C\otimes H^{\otimes -2}\).
While \(W^1_4C\) has the expected dimension, it is not irreducible, and it is singular even though \(W^2_4C = \emptyset\).  Furthermore, every line bundle of the form \(H(p)\) has a basepoint, even though \(r=1\).
\end{example}

H.~Larson \cite{l_rbn} and Cook-Powell--Jensen \cite{cpj1, cpj2} independently suggested that this behavior might be explained by the isomorphism class of the vector bundle \(f_*L\).  Since every vector bundle on \(\pp^1\) splits as a direct sum of line bundles, this isomorphism class is explicitly a \(k\)-tuple of integers \(\vec{e} \colonequals (e_1, \dots, e_k)\), which we assume is ordered so \(e_1 \leq e_2 \leq \cdots \leq e_k\), termed the \defi{splitting type} of the line bundle \(L\) on \(C\).  
Returning to Example~\ref{ex:trigonal_baby}, a line bundle \(L\in \Pic^4_C\) is in \(X_1\) if and only if \(h^0(C, L(-H)) = h^0(\pp^1, f_*L(-1)) \geq 0\).  Thus \(X_1\) is the locus of line bundles whose largest component \(e_3\) of the splitting type is at least \(1\).
The splitting type \(\vec{e}\) gives a refinement of the pair \((r, d)\), using
\begin{align}\label{eq:rd}
d &= \chi(L) +g - 1 = \chi(f_*L) +g - 1=  k + \sum e_i +g - 1,\\
r &= \dim H^0(C, L) - 1 = \dim H^0(C, f_*L)-1 = \sum_i \max(0,e_i-1) -1.
\end{align}
The locus \(W^r_dC\) is therefore refined by the splitting loci:
\[W^{\vec{e}} C \colonequals \{[L] \in \Pic_C : f_*L \simeq \O(e_1)\oplus \dots\oplus \O(e_k) \text{ or a specialization thereof}\}.\]
Specialization is given by majorization: \(W^{\vec{e}'}C \subset W^{\vec{e}}C\) if \(e'_1 + \cdots + e'_j \leq e_1 + \cdots + e_j\) for all \(j\).
In this language, in Example~\ref{ex:trigonal_baby}, \(X_1 = W^{(-2, -2, 1)}C\), \(X_2 = W^{(-3, 0, 0)}C\) and their two points of intersection are \(W^{(-3, -1, 1)}C\).

The expected dimension of \(W^{\vec{e}}C\) can be computed from deformation theory to be
\[\rho(g, \vec{e}) \colonequals g - h^1(\pp^1, \End \O(\vec{e})) = g - \sum_{i > j} \max(0, e_i - e_j - 1).\]

\begin{thm}[The Hurwitz--Brill--Noether Theorem]\label{thm:hbn}
Fix \(g, k, \vec{e}\).
Let \(f\colon C \to \pp^1\) be a general degree \(k\) genus \(g\) cover (i.e., \([f \colon C \to \pp^1]\) is in a dense open in the Hurwitz space \(\H_{k,g}\)).
\begin{enumerate}
\item (H.~Larson \cite[Theorem 1.2]{l_rbn}, Cook-Powell--Jensen \cite[Theorem 1.1]{cpj1, cpj2}) \(W^{\vec{e}}C\) has dimension \(\min(g, \rho(g, \vec{e}))\).  In particular, \(W^{\vec{e}}C = \emptyset\) if \(\rho(g, \vec{e}) < 0\).\label{eq:hbn_dim}
\item (H.~Larson \cite[Theorem 1.2]{l_rbn}, E.~Larson--H.~Larson--V. \cite[Theorem 1.2(2')]{llv}) \(W^{\vec{e}}C\) is normal, Cohen--Macaulay, and smooth away from all \(W^{\vec{e}'}C\) for \(\vec{e}'\) a specialization of \(\vec{e}\) of codimension at least \(2\).
\item (E.~Larson--H.~Larson--V. \cite[Theorem 1.4 and Section 5.2]{llv}) When \(\rho(g, \vec{e}) = 0\),  \(\#W^{\vec{e}}C\) is the number of \(k\)-fillings of a certain \(k\)-core \(\Gamma(\vec{e})\) with symbols \(\{1, \dots, g\}\).\label{eq:hbn_enum}
\item (E.~Larson--H.~Larson--V. \cite[Theorem 1.2(4')-(5')]{llv}) When \(\rho(g, \vec{e}) > 0\), the locus \(W^{\vec{e}}C\) is irreducible.  When \(\rho(g, \vec{e}) = 0\), if the characteristic does not divide \(k\), there exists a unique component of the universal \(\W^{\vec{e}}\) dominating  \(\H_{k,g}\).
\end{enumerate}
\end{thm}

\begin{rem}
You will notice that Theorem~\ref{thm:hbn} parallels \textit{almost} all of the main results from the classic Brill--Noether Theorem.  For \defi{tame} splitting types, which are concatenations of two balanced splitting types, a recent theorem of Lin proves that \(W^{\vec{e}}C\) has rational singularities in characteristic \(0\) \cite[Theorem 1]{Lin}.  
When \(\rho(g, \vec{e}) = 0\), the monodromy of the cover \(\W^{\vec{e}}\to \H_{k,g}\) is studied in forthcoming work of the author and Gonz\'alez, H.~Larson, and Pechenik \cite{mon}.
\end{rem}

Part~\eqref{eq:hbn_enum} hints at the deep combinatorial structure of this problem.  A \(k\)-core is a certain type of Young diagram whose boundary satisfies a \(k\)-discrete convexity property, see \cite{lm} and \cite[Section 5.1]{llv} for precise definitions and properties.  
In a \(k\)-filling of a \(k\)-core, a symbol can be repeated in boxes that are lattice distance a multiple of \(k\) apart.
Cores are well-studied in combinatorics for their relationship with the affine symmetric group \(\tilde{S}_k\), the infinite Coxeter group generated by transpositions \(s_0, \dots, s_{k-1}\) such that \(s_is_j = s_js_i\) if \(j \not\equiv i \pm 1 \pmod{k}\) and \(s_i s_{i+1}s_i = s_{i+1}s_is_{i+1}\).  In particular, \(k\)-filings are in bijections with reduced words in the affine symmetric group , see \cite[Section 8]{lm}.

\begin{example} The two \(3\)-fillings of the \(3\)-core \(\Gamma((-3, -1, 1))\) with symbols \(\{1, \dots, 5\}\).  

\begin{minipage}{.48\textwidth}
\begin{center}
\begin{tikzpicture}[scale=.5]

\draw (0,0) -- (4,0);
\draw (0, -1) -- (4,-1);
\draw (0, -2) -- (2, -2);
\draw (0, -3) -- (1, -3);
\draw (0, -4) -- (1, -4);

\draw (0,0) -- (0, -4);
\draw (1,0) -- (1, -4);
\draw (2, 0) -- (2, -2);
\draw (3, 0)--(3, -1);
\draw (4,0)--(4, -1);

\node at (.5, -.5) {\(1\)};
\node at (1.5, -.5) {\(2\)};
\node at (.5, -1.5) {\(3\)};
\node at (2.5, -.5) {\(3\)};
\node at (.5, -2.5) {\(4\)};
\node at (.5, -3.5) {\(5\)};
\node at (1.5, -1.5) {\(5\)};
\node at (3.5, -.5) {\(5\)};

\node at (5, -2) {\(\leftrightarrow\)};
\node at (8, -2) {\(s_0s_1s_2s_1s_0\)};
\end{tikzpicture}
\end{center}
\end{minipage}
\begin{minipage}{.48\textwidth}
\begin{center}
\begin{tikzpicture}[scale=.5]

\draw (0,0) -- (4,0);
\draw (0, -1) -- (4,-1);
\draw (0, -2) -- (2, -2);
\draw (0, -3) -- (1, -3);
\draw (0, -4) -- (1, -4);

\draw (0,0) -- (0, -4);
\draw (1,0) -- (1, -4);
\draw (2, 0) -- (2, -2);
\draw (3, 0)--(3, -1);
\draw (4,0)--(4, -1);

\node at (.5, -.5) {\(1\)};
\node at (1.5, -.5) {\(3\)};
\node at (.5, -1.5) {\(2\)};
\node at (2.5, -.5) {\(4\)};
\node at (.5, -2.5) {\(3\)};
\node at (.5, -3.5) {\(5\)};
\node at (1.5, -1.5) {\(5\)};
\node at (3.5, -.5) {\(5\)};

\node at (5, -2) {\(\leftrightarrow\)};
\node at (8, -2) {\(s_0s_2s_1s_2s_0\)};
\end{tikzpicture}\\
\end{center}
\end{minipage}

\noindent
This agrees with \(\#W^{(-3, -1, 1)}C = 2\) observed in Example \ref{ex:trigonal_baby}.\hfill\(\righthalfcup\)
\end{example}

Going back to Question~\ref{ques:spec}, the Hurwitz--Brill--Noether Theorem~\ref{thm:hbn} can be used to understand the geometry of \(W^r_dC\) when \(C\) is a general curve in \(\M^1_{g,k}\).
The components of \(W^r_dC\) are splitting loci \(W^{\vec{e}}\) for \(\vec{e}\) satisfying~\eqref{eq:rd} that are maximal with respect to containment.  Such splitting types are always tame and are explicitly given in \cite[Lemma 2.2]{l_rbn}; if \(b(x, y)\) denotes the unique balanced splitting type of length \(x\) and sum \(y\), then the maximal splitting types (when \(g-d + r >0 \)) are:
\[\vec{w}_{r, \ell} \colonequals b(k - r- 1 +\ell, d - g +1-k-\ell)  \sqcup b(r+1-\ell, \ell),\]
for \(\max(0, r+2-k) \leq \ell \leq r\) and \(\ell =0\) or \(\ell \leq g -d + 2r+1-k\).  
This yields the following corollary (we have omitted statements that are not substantially simpler than the Hurwtiz--Brill--Noether Theorem, and we give original references).

\begin{cor}\label{cor:hbn_wrd}
Let \(C \in \M^1_{g,k}\) be a general curve.
\begin{enumerate}
\item\label{hbn_dim} (Pflueger \cite[Theorem 1.1]{Pf} and Jensen--Ranganathan \cite[Theorem A]{jr}) \(W^r_d C\) has dimension 
\(\min(g, \rho_k(g, r, d))\), where
\[\rho_k(g, r, d) \colonequals  \max_{0 \leq \ell \leq \min(r, g-d+r-1)} \rho(g, r- \ell, d) - \ell k.\]
In particular, if \(\rho_k(g, r, d) < 0\), then \(W^r_dC = \emptyset\). 
\item (Lin \cite[Corollary 1]{Lin}) When the characteristic is \(0\), each component of \(W^r_d C\) has rational singularities.
\end{enumerate}
\end{cor}

The references given above for the upper bound on the dimension in part~\eqref{hbn_dim} of Corollary~\ref{cor:hbn_wrd} all involved degeneration --- either to a chain of elliptic curves \cite{l_rbn, llv}, or via tropical geometry \cite{Pf, cpj1, cpj2} ---  in some step of the proof.  Recently, Farkas--Feyzbakhsh--Rojas \cite[Theorem 1.4]{ffr} gave a new proof that \(\dim W^r_dC = \min(g, \rho_k(g, r, d))\) using Bridgeland stability and elliptic K3 surfaces, which has the feature of producing \textit{smooth} curves that are Hurwitz--Brill--Noether general.

Hurwitz--Brill--Noether theory can be used to explain \(\rho\)-unexpected linear series and components of large dimension.  We illustrate this in two examples.

\begin{example}\label{ex:814}[{cf.~\cite[Lemma 3.8]{mukai}}]
Let \(C\) be a general \(4\)-gonal curve of genus \(8\).  Since \(\rho(8, 2, 7) = -1\), the Brill--Noether locus \(\M^2_{8, 7}\) has codimension \(1\).  But, by Corollary~\ref{cor:hbn_wrd}, we have \(\dim W^2_7C = 0\), and so \(\M^1_{8,4} \subset \M^2_{8,7}\).  This explains the counterexample \((8,1,4)\) in Theorem~\ref{thm:max}.\hfill\(\righthalfcup\)
\end{example}

\begin{example}\label{ex:1123}[{cf.~\cite[Example 4.4]{ht_red}}]
Let \(C\) be general trigonal curve of genus \(12\).  
By Corollary~\ref{cor:hbn_wrd}, we have \(\dim W^2_7C = 1\), and so \(\M^1_{12,3} \subset \M^2_{12,7}\). 
The component \(\M^1_{12, 3}\) has codimension \(- \rho(12, 1, 3) = 8\), which is strictly less than \(-\rho(12, 2, 7) = 9\).
Pflueger's Theorem~\ref{thm:Pf_neg} shows that \(\M^2_{12,7}\) has a component of the expected codimension \(9\), and so there are at least two components.\hfill\(\righthalfcup\)
\end{example}

We now turn to the geometry of the maps to projective space corresponding to points in \(W^{\vec{e}}C\).  As we see from Example~\ref{ex:trigonal_baby}, it is no longer true that a general point of \(W^{\vec{e}}C\) corresponds to a basepoint free line bundle once \(r \geq 1\).  There is an extra condition coming from the number of nonnegative parts of the splitting type:

\begin{thm}[{Cook-Powell--Jensen--E.~Larson--H.~Larson--V. \cite[Theorems 1 and 2]{hemb}}]
Let \(f \colon C \to \pp^1\) be a general degree \(k\) genus \(g\) cover.
\begin{enumerate}
\item A general line bundle \(L \in W^{\vec{e}}C\) is basepoint free if and only if \(e_{k-1} \geq 0\) (which implies \(r \geq 1\)).
\item A general line bundle \(L \in W^{\vec{e}}C\) is very ample if \(e_{k-2} \geq 0\) and \(r \geq 3\).
\end{enumerate}
\end{thm}

 Necessary and sufficient criteria for very ampleness are given in \cite[Theorem 8.8]{hemb}.
We conjecture \cite[Conjecture 8.3]{hemb} that a general line bundle in \(W^{\vec{e}}C\) is \(p\)-very ample if \(e_{k-p-1} \geq 0\) and \(r \geq 2p+1\).

Extra obstructions to \(p\)-very ampleness for splitting loci exist because the image of \(C\) under the complete linear series of \(L \in W^{\vec{e}}\) lies on a scroll \(\pp E^\vee \xrightarrow{|\O_{\pp E^\vee}(1)|} \pp H^0(C, L)^\vee\), where \(E \simeq \bigoplus_{i : e_i \geq 0} \O_{\pp^1}(e_i)\) is the nonnegative part of \(\O(\vec{e})\).

\begin{example}
Suppose that \(C\) is a general curve of genus \(5\), which is tetragonal.  A line bundle of degree \(10\) always gives a map to \(\pp^5\), and the general such line bundle is \(2\)-very ample by the Embedding Theorem~\ref{thm:emb}.  Let \(f \colon C \to \pp^1\) be a degree \(4\) map and let \(L \in W^{(-1, 0, 1, 2)}C\), which is an irreducible \(1\)-dimensional locus in \(\Pic^{10}_C\) by the Hurwitz--Brill--Noether Theorem~\ref{thm:hbn}.  Since \(H^0(C, L) = H^0(\pp^1, f_*L) = H^0(\pp^1, E \colonequals \O_{\pp^1}\oplus\O_{\pp^1}(1)\oplus \O_{\pp^1}(2))\), the complete linear series \(\varphi_L\) of \(L\) factors through a map \(C \to \pp E^\vee\):
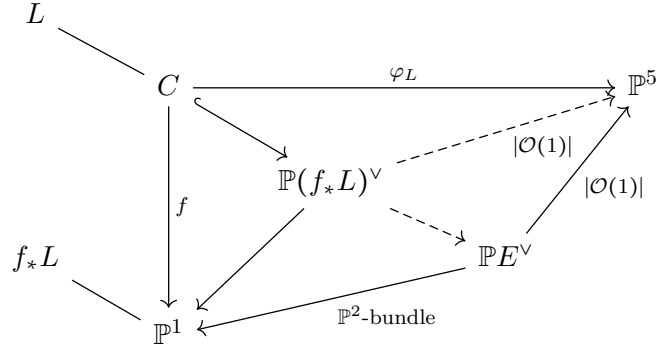
\begin{center}
\begin{tikzcd}[row sep = small]
L \arrow[dr, dash] & \\
&C \arrow[dddd, "f"] \arrow[rrr, "\varphi_L"] \arrow[ddr, hook] &&& \pp^5 \\
&\\
&& \pp(f_*L)^\vee \arrow[ddl] \arrow[uurr, swap, dashed, "|\O(1)|"] \arrow[dr, dashed] \\
f_*L \arrow[dr, dash] &&& \pp E^\vee \arrow[uuur,swap,  "|\O(1)|"] \arrow[dll, "\pp^2\text{-bundle}"]\\
&\pp^1
\end{tikzcd}
\end{center}
The image of \(C\) in the \(\pp^2\)-bundle \(\pp E^\vee\) meets every fiber in \(4\) points, and each fiber is linearly embedded in \(\pp^5\) by \(|\O_{\pp E^\vee}(1)|\).
Since it is a divisorial condition for \(4\) points in \(\pp^2\) to contain \(3\) collinear points, you would expect that some fiber of \(f\) to contain \(3\) collinear points.  In \cite[Section 7.1.1]{hemb} we compute that there are exactly \(2\) such fibers counted with multiplicity.  So \(L\) is \textbf{not} \(2\)-very ample even though \(r \geq 5\).\hfill\(\righthalfcup\)
\end{example}

\subsection{Geometry of Brill--Noether curves}

We now turn to understanding the geometry of the curves \(f \colon C \to \pp^r\) parameterized by the 
Brill--Noether component \(\M_g(\pp^r, d)^{\BN}\).  Recall that for \(r \geq 3\), which we assume in this section, the general such map is an embedding.  We will often conflate the map \(f\) and its image \(C \subset \pp^r\).  We call such curves Brill--Noether curves, or \defi{BN-curves} for short.  The most fundamental question for any subvariety of projective space is the following:

\begin{question}\label{ques:ideal} 
Given a general BN-curve \(C \subset \pp^r\) of degree \(d\) and genus \(g\), what is the shape of its homogeneous ideal \(I_C\)? 
\end{question}

The first subquestion is what is the dimension of each graded piece of \(I_C\), i.e., what is the Hilbert function of \(C\)?  If we write \(\O_C(k) \colonequals \O_{\pp^r}(k)|_C\), then the \(k\)-graded piece of \(I_C\) is the kernel of the restriction map
\begin{equation}\label{eq:mrc}
H^0(\pp^r, \O_{\pp^r}(k)) \to H^0(C, \O_C(k)).
\end{equation}

The natural conjecture is that this map is of maximal rank.  This was first conjectured by Max Noether \cite{noether_mrc} and later reformulated in modern terms by Severi in 1915 \cite{severi_mrc}.  Special cases of this conjecture attracted significant attention \cite{mrat, bp, t_mrc,   b_quad, jp_mr, lotz1}. The full conjecture was recently proven in full generality in characteristic \(0\) by E.~Larson.

\begin{thm}[{The Maximal Rank Theorem (E.~Larson \cite[Theorem 1.2]{mrc})}]\label{thm:mrc}
Assume that the characteristic is \(0\).
Let \(C\subset \pp^r\) be a general BN-curve of degree \(d\) and genus \(g\).  Then for all \(k \geq 1\), the restriction map~\eqref{eq:mrc} is either injective or surjective.  In particular, the Hilbert function of \(C\) is \(\min\left({k+r \choose k}, kd+ 1 -g\right)\).
\end{thm}

The Strong Maximal Rank Conjecture \eqref{eq:mrc} is a substantial generalization of the Maximal Rank Conjecture, which predicts the dimension of the locus of line bundles in \(W^r_dC\) for a general curve \(C\) for which the Maximal Rank Conjecture holds.  To set notation, let \(L \in W^r_dC\).  If \(C\) is embedded in \(\pp^r\) by the complete linear series of \(L\) then the restriction map in~\eqref{eq:mrc} is the multiplication map
\begin{equation}\label{eq:smrc}
\phi^k_L \colon \Sym^k H^0(C, L) \to H^0(C, L^{\otimes k}).
\end{equation}
When \(g -d + r \geq 0\) and \(0 \leq \rho(g, r, d) < r-2\), every line bundle is guaranteed to be very ample \cite[Corollary 0.4]{farkas_p} and have \(h^0(C, L) = r+1\) (since \(\rho(g, r+1,d) <0\)).

\begin{conj}[{The Strong Maximal Rank Conjecture \cite[Conjecture 5.4]{af_smrc}}]\label{conj:smrc}
Let \(g, r, d, k\) be such that \(g-d+r \geq 0\) and \(0 \leq \rho(g, r, d) < r-2\) and \(k \geq 2\).  Let \(C\) be a general curve of genus \(g\).  Then the locus of line bundles \(L \in W^r_dC\) for which \(\phi^k_L\) is not of maximal rank has dimension
\begin{equation}\label{eq:smrc_dim}
\rho(g, r, d) - 1  - \left| {r + k \choose k} - (dk+1 - g) \right|.
\end{equation}
In particular,  \(\phi^k_L\) is of maximal rank for \textit{every} \(L \in W^r_dC\) when this quantity is negative.
\end{conj}

\noindent The dimension appearing in~\eqref{eq:smrc_dim} is the expected dimension of the degeneracy locus in \(W^r_dC\).

While Conjecture~\ref{conj:smrc} remains very much open in general, recent results of Farkas--Jensen--Payne proves three cases by tropical methods.  The second two cases were independently obtained by Liu--Osserman--Teixidor i Bigas--Zhang.

\begin{thm}[{\cite[Theorem 1.5]{fjp2} for \(g=13\), and \cite[Theorem 1.3]{fjp} and \cite[Theorem 1.1]{lotz} for \(g=22,23\)}]\label{thm:fjp}
The Strong Maximal Rank Conjecture~\ref{conj:smrc} is true for 
\[(g, r, d, k) \in \{(13, 5, 16, 2), (22, 6, 25, 2), (23, 6, 26, 2)\}.\]
\end{thm}

In all of these cases, the expected dimension~\eqref{eq:smrc_dim} is \(-1\), and so Theorem~\ref{thm:fjp} proves that \eqref{eq:smrc} is of maximal rank for every \(L \in \Pic^d_C\) for general \(C\).  Since the expected dimension is \(-1\), one also expects that the locus of curves in \(\M_g\) for which \eqref{eq:smrc} fails to be of maximal rank is an effective divisor.
With substantial additional work, Farkas--Jensen--Payne prove this, and in the second and third cases that the closures of these divisors in \(\bar{\M}_{22}\) and \(\bar{\M}_{23}\) have slope less than \(13/2\), the slope of the canonical divisor \(K_{\bar{\M}_g}\).  This proves that the moduli spaces \(\bar{\M}_{22}\) and \(\bar{\M}_{23}\) are of general type \cite[Theorem 1.1]{fjp}.  In the first case, 
Farkas--Jensen--Payne use this to show that \(\bar{\M}_{13,n}\) is of general type for \(n \geq 9\) \cite[Theorem 1.6]{fjp2} and that the Prym moduli space \(\bar{\R}_{13}\) is of general type \cite[Theorem 1.2]{fjp2}.  

We end with a short overview of one of the key results going into E.~Larson's proof of the Maximal Rank Theorem~\ref{thm:mrc}.  The proof uses embedded degeneration to reducible nodal curves \(X \cup Y \subset \pp^r\), substantially refining an idea first utilized by Hirschowitz to prove the maximal rank conjecture for rational space curves \cite{mrat}.  To handle all cases, Larson utilized partial progress on the interpolation problem for BN-curves, which naturally builds reducible curves by taking the union of two curves passing through the same collection of points.  
To state this problem, let \(\M_{g,n}(\pp^r,d)^{\BN}\) be the unique component of \(\M_{g,n}(\pp^r, d)\) mapping to the Brill--Noether component under the natural map \(\M_{g,n}(\pp^r, d) \to \M_g(\pp^r, d)\).  An element of \(\M_{g,n}(\pp^r, d)\) is a pair \([f \colon C \to \pp^r, p_1, \dots, p_n]\), where \(p_i \in C\) are distinct points.  There is a natural evaluation map
\begin{equation}\label{eq:eval}
\ev_n \colon \M_{g,n}(\pp^r, d)^{\BN} \to (\pp^r)^n,  \qquad [f \colon C \to \pp^r, p_1, \dots, p_n] \mapsto (f(p_1), \dots, f(p_n)).
\end{equation}
The interpolation problem for Brill--Noether curves asks:

\begin{question}\label{ques:inter}
For given \((g, r, d)\), what is the largest \(n\) such that \(\ev_n\) is dominant?
\end{question}

Concretely, this asks for the maximum number of general points interpolated by a general BN-curve of fixed invariants.

By the Brill--Noether Theorem (cf.~Corollary~\ref{cor:BN_comp}) the dimension of \(\M_{g,n}(\pp^r,d)^{\BN}\) is \((r+1)d - (r-3)(g-1) + n\).  A dimension count therefore predicts that the answer should be
\begin{equation}\label{eq:inter_conj}
\text{\(\ev_n\) is dominant} \qquad \Leftrightarrow \qquad n \leq \left\lfloor \frac{(r+1)d - (r-3)(g-1)}{r-1} \right\rfloor.
\end{equation}

However this expectation cannot always hold: for curves of degree \(5\) and genus \(2\) in \(\pp^3\), \eqref{eq:inter_conj} predicts that they interpolate \(10\) general points; but every such curve is contained in a quadric surface by comparing dimensions in~\eqref{eq:mrc}, and it is elementary linear algebra that quadric surfaces interpolate only \(9\) general points.  Similar counterexamples exist curves of degree \(7\) and genus \(2\) in \(\pp^5\) and for canonical curves in \(\pp^3\) and \(\pp^5\).  Nevertheless, many authors contributed partial work \cite{sacc_int, can, ran_int,  p3, aly, ibe, p4}.  A full solution in arbitrary characteristic was recently obtained in joint work with E.~Larson.

\begin{thm}[{E.~Larson--V. \cite[Theorem 1.2]{interpolation}}]\label{thm:interpolation}
Brill--Noether curves of degree \(d\) and genus \(g\) in \(\pp^r\) interpolate the expected number \(\left\lfloor \frac{(r+1)d - (r-3)(g-1)}{r-1} \right\rfloor\) of general points if and only if
\[(g, r, d) \not\in \{(2, 3, 5), (4, 3, 6), (2, 5, 7), (6, 5, 10)\}.\]
\end{thm}

This theorem follows by deformation theory from a semistability-like property of the normal bundles \(N_C\) of general BN-curves \(C\) \cite[Theorem 1.4]{interpolation}.  The techniques developed in \cite{interpolation} (see Section~\ref{sec:normal} below) can be refined to prove (semi)stability of normal bundles for space curves \cite[Theorem 1]{stab_p3}, canonical curves \cite[Theorem 1.1]{canonical} (resolving the semistability version of a conjecture of Aprodu--Farkas--Ortega \cite[Conjecture 0.4]{AFO}), and asymptotically when \(d\) is large relative to \(r\) \cite[Theorems 1.1 and 1.2]{cs_norm}.

\section{Important Techniques}\label{sec:techniques}

\subsection{Abstract degeneration and limit line bundles}\label{sec:BN_nonexist}
Degeneration is a fruitful tool in all aspects of Brill--Noether theory.  We begin with the technique of degenerating the curve \(C\) and keeping track of how the line bundles in \(W^r_dC\) degenerate.  
To illustrate this technique we give a short proof of the classic Brill--Noether nonexistence theorem.

\begin{thm}[Brill--Noether Nonexistence]\label{thm:BN_nonexistence}
Suppose that \(\rho(g, r, d) < 0\).  If \(C\) is general curve of genus \(g\), then \(W^r_dC\) is empty.
\end{thm}

Unwinding definitions, we must show that for every line bundle \(L\) on \(C\) of degree \(d\), we have \(h^0(C, L) \leq r\).
The key idea here is to degenerate \(C\) to a nodal curve \(X\); if we can show that \(L\) admits a limit on \(X\) with \(h^0\) at most \(r\), then the result follows from the semicontinuity theorem.  One might hope to show this by showing that \textit{every} line bundle on \(X\) of degree \(d\) has at most \(r\) global sections.
There are two potential subtleties in this approach: (1) If \(X\) is an arbitrary nodal curve, then \(L\) may limit to a torsion free sheaf that is not a line bundle; (2) Even if \(L\) limits to a line bundle on \(X\), this limit is not necessarily unique.

We sidestep the first issue by considering \defi{compact type} \(X\) --- a connected nodal curve whose dual graph is a tree, so called because the identity component of the Picard scheme is compact --- which guarantees that \(L\) limits to a line bundle on \(X\).  The second issue is actually a feature not a bug, and it is the crucial tool we will use: with more possible limits, it is \textit{easier} to find some limit with at most \(r\) sections.

The first proof of the Brill--Noether nonexistence theorem by Griffiths--Harris \cite{gh_80} used degeneration to a \(g\)-nodal curve and a subtle analysis of Schubert cycles in a further degeneration.
A second proof by Eisenbud--Harris \cite{EH83} utilized degeneration to a \(g\)-cuspidal curve, whose stable reduction is a rational curve with \(g\) elliptic tails (called a \defi{flag curve}) that was used in a subsequent reproof by also by Eisenbud--Harris \cite{EH_gp}.  Degeneration to a chain of \(g\) elliptic curves meeting at general nodes first appears in the work of Welters on Prym--Brill--Noether theory \cite{welters}, and there is now consensus that this degeneration yields clean and characteristic-free proofs.

Let \(X\) denote a chain of \(g\) elliptic curves \(E^1 \cup \cdots \cup E^g\) 
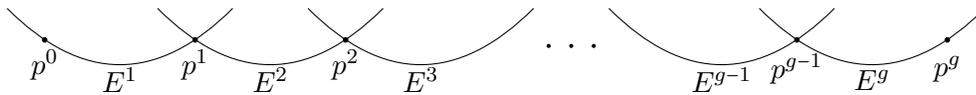
\begin{center}
\begin{tikzpicture}
\draw (0, 0) .. controls (1, -1) and (2, -1) .. (3, 0);
\draw (2, 0) .. controls (3, -1) and (4, -1) .. (5, 0);
\draw (4, 0) .. controls (5, -1) and (6, -1) .. (7, 0);
\filldraw (7.8, -0.5) circle[radius=0.02];
\filldraw (7.5, -0.5) circle[radius=0.02];
\filldraw (7.2, -0.5) circle[radius=0.02];
\draw (8, 0) .. controls (9, -1) and (10, -1) .. (11, 0);
\draw (10, 0) .. controls (11, -1) and (12, -1) .. (13, 0);
\filldraw (0.5, -0.42) circle[radius=0.03];
\filldraw (2.5, -0.42) circle[radius=0.03];
\filldraw (4.5, -0.42) circle[radius=0.03];
\filldraw (10.5, -0.42) circle[radius=0.03];
\filldraw (12.5, -0.42) circle[radius=0.03];
\draw (0.5, -0.75) node{$p^0$};
\draw (2.5, -0.75) node{$p^1$};
\draw (4.5, -0.75) node{$p^2$};
\draw (10.5, -0.8) node{$p^{g - 1}$};
\draw (12.5, -0.8) node{$p^g$};
\draw (1.5, -0.95) node{$E^1$};
\draw (3.5, -0.95) node{$E^2$};
\draw (5.5, -0.95) node{$E^3$};
\draw (9.5, -0.95) node{$E^{g - 1}$};
\draw (11.5, -0.95) node{$E^g$};
\end{tikzpicture}
\end{center}
such that the points \(p^{i-1}\) and \(p^i\) on \(E^i\) are general.  A line bundle on \(X\) is equivalent to a line bundle on each component \(E^i\).  The discrete invariant of such a line bundle is therefore a tuple of degrees \(\vec{d} = (d^1, d^2, \dots, d^g)\), where \(d^i\) is the degree of the bundle on the \(i\)th component.  The \defi{total degree} of such a line bundle is \(|\vec{d}| \colonequals \sum_i d^i\).

Consider a \(1\)-parameter degeneration \(\C \to \Spec K\llbracket t \rrbracket\) with general fiber \(C\) a smooth curve of genus \(g\) and special fiber \(X\) as above and such that the total space \(\C\) is smooth over the ground field \(K\) (such a \(1\)-parameter degeneration exists by formal patching \cite[Lemma 5.6]{liu}).  A line bundle \(L\) on \(C\) extends to a line bundle \(\L\) on \(\C\) (for example, write it as a difference of effective divisors on \(C\) and then take the closures of these divisors on \(\C\); since \(\C\) is smooth, this is necessarily Cartier).  The restriction \(\L|_X\) is one such limit of \(L\) to \(X\).  This limit has some degree distribution \(\vec{d}(\L|_X) = (d^1, \dots, d^g)\) such that \(|\vec{d}| = d = \deg(L)\).  However, given any nontrivial divisor \(D\) supported on the central fiber --- i.e., a linear combination of the components of \(X\) --- the line bundle \(\L(D)|_X\) is another limit of \(L\) to \(X\).

To be explicit, consider the twist by a single component \(E^i\).  Then, using the fact that \(\O_{\C}(E^1 + \cdots + E^g) \simeq \O_{\C}\), we have
\begin{equation}\label{eq:twist}
\L(E^i)|_{E^j} \simeq \begin{cases} \L|_{E^{i-1}}(p^{i-1}) & : \  j = i-1 \\
\L|_{E^{i}}(-p^{i-1} - p^i) & : \  j = i \\
\L|_{E^{i+1}}(p^{i}) & : \  j = i+1 \\
\L|_{E^{j}} & : \  \text{else}.
 \end{cases}
 \end{equation}
In particular, the limit \(\L(E^i)|_X\) can be distinguished from the limit \(\L|_X\) by its degree distribution; if \(\vec{d}(\L|_X) = (d^1, d^2, \dots, d^g)\) then 
\begin{equation}\label{eq:chip_fire}
\vec{d}(\L(E^i)|_X) = (d^1, \dots, d^{i-2}, d^{i-1} + 1, d^i - 2, d^{i+1}+1, d^{i+2},\dots, d^g).
\end{equation}
Combinatorially, the procedure of twisting by the \(i\)th component as in \eqref{eq:twist} has the effect on the degree distribution of \defi{chip firing} the \(i\)th vertex in the dual graph as in \eqref{eq:chip_fire}.  In tropical degenerations, where the limit is not compact type, but in fact every component has genus \(0\), chip firing generates the linear equivalence relation for divisors on graphs.
We will not explore this connection here, and only point it out so that the interested reader can see the parallels.

In fact, every possible limit of \(L\) on \(C\) to \(X\) can be obtained from one limit \(\L|_X\) by a sequence of chip firings as in \eqref{eq:twist}.  Induction on \(g\) easily proves that every possible degree distribution \(\vec{d}\) with total degree \(d\) can be obtained in \eqref{eq:chip_fire} in this way.
The upshot is that given any degree distribution, there is a unique limit having that degree distribution, and the limit for any one degree distribution determines all the others. We will denote the limit with a given degree distribution \(\vec{d}\) by \(L_{\vec{d}}\).  This can be easily computed by the combinatorial process of chip firing, as in the following example (that will be our running example for this section).

\begin{example}\label{ex:running_one_lim}
Suppose that \(g=3\) and \(d = 4\) and that \(\L|_X\) has the following components:
\[\L|_{E^1} \simeq \O_{E^1}(4p^1), \qquad \L|_{E^2} \simeq \O_{E^2}(-2p^1 + 2p^2), \qquad \L|_{E^3} \simeq \O_{E^1}.\]
In this case \(\vec{d}(\L|_X) = (4, 0, 0)\).  The limit with degree distribution \((3, 0, 1)\) is obtained from \(\L|_X\) by firing the first component twice and the second component once (i.e., it is \(\L(2E^1 + E^2)|_X\)), and using \eqref{eq:twist} this has components:
\[L_{(3,0,1)}|_{E_1} \simeq \O_{E^1}(3p^1), \qquad L_{(3,0,1)}|_{E_2} \simeq \O_{E^2}(-p^1 + p^2), \qquad L_{(3,0,1)}|_{E_3} \simeq \O_{E^3}(p^2).\]
Similarly, the limit with degree distribution \((1, 2, 1)\) has components:
\[L_{(1, 2,1)}|_{E_1} \simeq \O_{E^1}(p^1), \qquad L_{(2,1,1)}|_{E_2} \simeq \O_{E^2}(p^1 + p^2), \qquad L_{(1, 2,1)}|_{E_3} \simeq \O_{E^3}(p^2).\]
\hfill\(\righthalfcup\)
\end{example}

While this approach of recording any one limit suffices to determine all possible limits, it unnecessarily breaks the symmetry in the problem by singling out one degree distribution.  We now describe an equivalent and more symmetric way of recording the information of all possible limits.  For each \(1 \leq i \leq g\), let \(\vec{d}_i \colonequals (0,  \dots, 0, d, 0, \dots, 0)\) be the degree distribution with degree concentrated on the \(i\)th component. Define  
\[L^i \colonequals L_{\vec{d}_i}|_{E^i} \in \Pic^{d}_{E^i}.\]
The tuple \((L^1, \dots, L^g)\) is equivalent to the data of any one limit, via the process of chip firing~\eqref{eq:twist}. 

\begin{example}\label{ex:running_aspects}
Continuing with the line bundle in Example~\ref{ex:running_one_lim}, we compute that
\begin{align*}
L^1 &= L_{(4,0,0)}|_{E^1} \simeq \O_{E^1}(4p^1), \\ 
L^2 &= L_{(0,4,0)}|_{E^2} = L_{(4,0,0)}|_{E^2}(4p^1) \simeq \O_{E^2}(2p^1 + 2p^2), \\
L^3 &= L_{(0,0,4)}|_{E^3} = L_{(4,0,0)}|_{E^2}(4p^2) \simeq \O_{E^3}(4p^2).
\end{align*}
\hfill\(\righthalfcup\)
\end{example}

\begin{defin}
A \defi{limit line bundle} of degree \(d\) on \(X\) is a tuple
\(\vec{L}  = (L^1, \dots, L^g) \in \Pic^d_{E^1} \times \cdots \times \Pic^d_{E^g}\). The \(i\)th component \(L^i\) is called the \(i\)th \defi{aspect} of the limit line bundle.
\end{defin}

\begin{warning}
A limit line bundle is not, in itself, a limit of any line bundle of degree \(d\) --- in particular, since it has total degree \(dg\) not \(d\) --- but rather is a convenient and symmetric way of recording all possible limits. 
\end{warning}

Given a limit line bundle \(\vec{L}\) of degree \(d\) and a degree distribution \(\vec{d} = (d^1, \dots, d^g)\) with \(|\vec{d}| = d\), we write \(L_{\vec{d}}\) for the line bundle of multidegree \(\vec{d}\) obtained from \(\vec{L}\) by the chip firing procedure in~\eqref{eq:twist}.

\begin{defin}\label{def:r-pos}
A limit line bundle \(\vec{L}\) on \(X\) is \defi{\(r\)-positive} if for all degree distributions \(\vec{d}\) with \(|\vec{d}| = d\) we have \(h^0(X, L_{\vec{d}}) \geq r+1\).
Define the space of space of \(r\)-positive limit line bundles of degree \(d\):
\[LW^r_dX \colonequals \left\{\vec{L} \in \prod_i \Pic^d_{E^i} : \vec{L} \text{ is \(r\)-positive} \right\}.\]
\end{defin}

The space \(LW^r_dX\) acquires a scheme structure as an intersection of determinantal loci.  
%
%

\begin{example}
The limit line bundle in our running example whose aspects are given in Example~\ref{ex:running_aspects} is \(2\)-positive; we will not verify this entire claim here, but rather check the required sections for the degree distributions \((4, 0, 0)\), \((3, 0, 1)\), and \((1, 2, 1)\) whose components were computed in Example~\ref{ex:running_one_lim}.

For degree distribution \((4, 0, 0)\), all sections must vanish identically on \(E^2\).  Hence, considering agreement at the point \(p^2\), we see that the sections must also vanish along \(E^3\).  Finally, considering agreement at the point \(p^1\), we have
\[H^0(X, L_{(4, 0, 0)}) \simeq H^0(E^1, \O_{E^1}(4p^1)(-p^1)) \simeq k^{\oplus 3}.\]
For degree distribution \((3, 0, 1)\), similarly all sections must vanish along \(E^2\).  But in this case, since \(p^2\) is a basepoint of \(L_{(3, 0, 1)}|_{E^3}\simeq \O_{E^3}(p^2)\), there is a \(1\)-dimensional space of sections on \(E^3\).  And there is a \(2\)-dimensional space of sections along \(E^1\) isomorphic to \(H^0(E^1, \O_{E^1}(3p^1)(-p^1))\).  All together, \(h^0(X, L_{(3, 0, 1)}) = 3\) as desired.
Finally, for degree distribution \((1, 2, 1)\), there are \(1\)-dimensional spaces of sections \(H^0(E^1, \O_{E^1}(p^1))\) and \(H^0(E^1, \O_{E^3}(p^2))\) on \(E^1\) and \(E^3\) respectively.  These have \(p^1\) and \(p^2\) respectively as basepoints.  Thus, to satisfy agreement conditions at the nodes, the
sections on \(E^2\) are necessarily \(H^0(E^2, \O_{E^2}(p^1 + p^2)(-p^1-p^2))\), which is also \(1\)-dimensional.  All together, again we have  \(h^0(X, L_{(1, 2, 1)}) = 3\).\hfill\(\righthalfcup\)
\end{example}

If \(\vec{L}\) is the limit line bundle corresponding to the line bundle \(\L|_X\) --- which is a limit of \(L\) from \(C\) --- then for any degree distribution \(\vec{d}\), the line bundle \(L_{\vec{d}}\) is again a limit of \(L\).  If \(h^0(C, L) \geq r+1\), then \(h^0(X, L_{\vec{d}}) \geq r+1\) by the semicontinuity theorem.  Hence, if \(LW^r_dX = \emptyset\), then \(W^r_dC = \emptyset\) as well.  This will be our technique for proving the Brill--Noether Nonexistence Theorem.
The key definition is the following.

\begin{defin}
Let \(\vec{L}\) be an \(r\)-positive limit line bundle of degree \(d\).  For each \(1 \leq i \leq g-1\) and \(0 \leq n \leq r\) define
\(a^i_n \colonequals a^i_n(\vec{L})\) to be the maximum \(\alpha \in \zz\) such that for all \(\vec{d}\) with \(|\vec{d}| = d\) and with \(\sum_{j=1}^i d^j = \alpha\), we have 
\[h^0(X^{>i}, L_{\vec{d}}|_{X^{>i}}) \geq r+1-n.\]
\end{defin}

\begin{lem}\label{lem:equality}
Let \(\vec{L}\) be an \(r\)-positive limit line bundle of degree \(d\).  Then there exists a degree distribution \(\vec{d}\) with \(|\vec{d}| = d\) and  \(\sum_{j=1}^i d^j = a^i_n(\vec{L})\) such that 
\[h^0(X^{>i}, L_{\vec{d}}|_{X^{>i}})  = r+ 1 - n, \qquad \text{and} \qquad  \ev_{p^i} \colon H^0(X^{>i}, L_{\vec{d}}|_{X^{>i}}) \to k \text{ is surjective.}\]
\end{lem}
\begin{proof}
If either \(h^0(X^{>i}, L_{\vec{d}}|_{X^{>i}})  > r+ 1 - n\) or \(\ev_{p^i} = 0\), then \(h^0(X^{>i}, L_{\vec{d}}|_{X^{>i}}(-p^i))  \geq  r+ 1 - n\).  Thus, if one of these always holds when \(\sum_{j=1}^i d^j = \alpha\), we can increase \(\alpha\) by one (which corresponds to firing all of \(X^{>i}\) which has the effect of twisting down by \(p^i\) on \(X^{>i}\) and twisting up by \(p^i\) on \(X^{\leq i}\)), contradicting the maximality of \(a^i_n(\vec{L})\).
\end{proof}

\begin{cor}\label{cor:increasing}
The sequence \(a^i_0(\vec{L}), a^i_1(\vec{L}), \dots, a^i_r(\vec{L})\) is strictly increasing. 
\end{cor}

\begin{example}
Returning to our running example of a \(2\)-positive limit line bundle whose aspects were computed in Example~\ref{ex:running_aspects}:
\[\vec{L} = (\O_{E^1}(4p^1), \O_{E^2}(2p^1 + 2p^2), \O_{E^3}(4p^2)) \in LW^2_4(X),\]
we will compute the sequence \(a^2_0(\vec{L}), a^2_1(\vec{L}), a^2_2(\vec{L})\).  By definition, we must consider degree distributions \(\vec{d}\) that have degree \(\alpha\) on \(E^1 \cup E^2\) and hence have degree \(d-\alpha\) on \(E^3\).  Since \(E^3\) is a genus \(1\) curve, for any such \(\vec{d}\), we have 
\begin{equation}\label{eq:h0_g1}
h^0(E^3, L_{\vec{d}}|_{E^3}) = \begin{cases} d-\alpha & : \ d-\alpha \geq 1, \\
1 & : \ L_{\vec{d}}|_{E^3} \simeq \O_{E^3}, \\
0 & : \text{ else.} \end{cases}
\end{equation}
Moreover, by the chip-firing procedure \eqref{eq:twist} for changing the degree distribution, we have \(L_{\vec{d}}|_{E^3} \simeq \O_{E^3}((d-\alpha)p^2)\).  Hence, combining \eqref{eq:h0_g1} with the definition of \(a^i_n(\vec{L})\) we have
\[a^2_0(\vec{L}) = 1, \qquad a^2_1(\vec{L}) = 2, \qquad a^2_2(\vec{L}) = 4.\]
\hfill\(\righthalfcup\)
\end{example}

\begin{defin}
Given an \(r\)-positive limit line bundle \(\vec{L}\) of degree \(d\), for \(1 \leq i \leq g-1\) and for \(0 \leq n \leq r\) we define
\begin{equation}\label{eq:b_complement}
b^i_n \colonequals b^i_n(\vec{L}) = d - a^i_{r-n}(\vec{L}).
\end{equation}
\end{defin}

Observe that Corollary~\ref{cor:increasing} guarantees that the \(b^i_0(\vec{L}), \dots, b^i_r(\vec{L})\) is an increasing sequence as well.  
It is useful to extend the definition of \(a^i_n(\vec{L})\) to \(i=0\) and the definition of \(b^i_n(\vec{L})\) to \(i=g\) by declaring \[a^0_n(\vec{L}) \colonequals n, \qquad b^g_n(\vec{L}) \colonequals n.\]

\begin{example}
For our running example \(\vec{L}\) from Example~\ref{ex:running_aspects},
the sequences \(a^i_n(\vec{L})\) and \(b^i_m(\vec{L})\) are given in the diagram below.  We record these invariants in increasing order with \(a^i_n\) to the right of the node \(p^i\) and \(b^i_m\) to the left of the node \(p^i\):
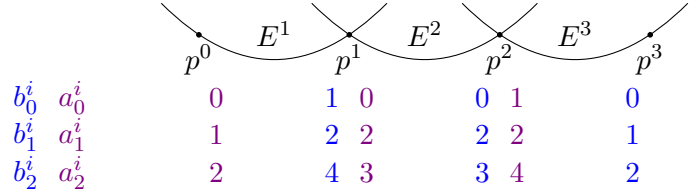
\begin{center}
\begin{tikzpicture}
\draw (0, 0) .. controls (1, -1) and (2, -1) .. (3, 0);
\draw (2, 0) .. controls (3, -1) and (4, -1) .. (5, 0);
\draw (4, 0) .. controls (5, -1) and (6, -1) .. (7, 0);
\filldraw (0.5, -0.42) circle[radius=0.03];
\filldraw (2.5, -0.42) circle[radius=0.03];
\filldraw (4.5, -0.42) circle[radius=0.03];
\filldraw (6.5, -0.42) circle[radius=0.03];
\draw (0.5, -0.75) node{$p^0$};
\draw (2.5, -0.75) node{$p^1$};
\draw (4.5, -0.75) node{$p^2$};
\draw (6.5, -0.75) node{$p^3$};
\draw (1.5, -0.4) node{$E^1$};
\draw (3.5, -0.4) node{$E^2$};
\draw (5.5, -0.4) node{$E^3$};
\draw (0.5, -1.25) node[right]{{\color{violet}\(0\)}};
\draw (0.5, -1.75) node[right]{{\color{violet}\(1\)}};
\draw (0.5, -2.25) node[right]{{\color{violet}\(2\)}};
\draw (2.5, -1.25) node[right]{{\color{violet}\(0\)}};
\draw (2.5, -1.75) node[right]{{\color{violet}\(2\)}};
\draw (2.5, -2.25) node[right]{{\color{violet}\(3\)}};
\draw (4.5, -1.25) node[right]{{\color{violet}\(1\)}};
\draw (4.5, -1.75) node[right]{{\color{violet}\(2\)}};
\draw (4.5, -2.25) node[right]{{\color{violet}\(4\)}};
\draw (2.5, -1.25) node[left]{{\color{blue}\(1\)}};
\draw (2.5, -1.75) node[left]{{\color{blue}\(2\)}};
\draw (2.5, -2.25) node[left]{{\color{blue}\(4\)}};
\draw (4.5, -1.25) node[left]{{\color{blue}\(0\)}};
\draw (4.5, -1.75) node[left]{{\color{blue}\(2\)}};
\draw (4.5, -2.25) node[left]{{\color{blue}\(3\)}};
\draw (6.5, -1.25) node[left]{{\color{blue}\(0\)}};
\draw (6.5, -1.75) node[left]{{\color{blue}\(1\)}};
\draw (6.5, -2.25) node[left]{{\color{blue}\(2\)}};
\draw (-1.5, -1.25) node[right]{{\color{violet}\(a^i_0\)}};
\draw (-1.5, -1.75) node[right]{{\color{violet}\(a^i_1\)}};
\draw (-1.5, -2.25) node[right]{{\color{violet}\(a^i_2\)}};
\draw (-1.5, -1.25) node[left]{{\color{blue}\(b^i_0\)}};
\draw (-1.5, -1.75) node[left]{{\color{blue}\(b^i_1\)}};
\draw (-1.5, -2.25) node[left]{{\color{blue}\(b^i_2\)}};
\end{tikzpicture}
\end{center}
By definition complementary \(b^i_n\) and \(a^i_{2-n}\) across a node \(p^i\) add to \(d=4\). \hfill\(\righthalfcup\)
\end{example}

The following is the key lemma in our proof of the Brill-Noether Nonexistence Theorem.

\begin{lem}\label{lem:lls}
If \(\vec{L} \in LW^r_dX\), then
\[h^0(E^i, L^i(-a^{i-1}_np^{i-1} - b^i_m p^i)) \geq r+1 - n - m.\]
\end{lem}

\begin{rem}
The notation \(a^i_0, \dots, a^i_r\) may remind you of the vanishing sequence of a linear series at a point, and this is not an accident!
Readers familiar with the theory of limit linear series will see that Lemma~\ref{lem:lls} guarantees that for any \(r\)-positive limit line bundle \(\vec{L}\), there exists a refined limit linear series with vanishing sequences \(a^i_n, b^i_m\) and underlying aspect line bundles \(\vec{L}\). 
\end{rem}

\begin{proof}[Proof of Lemma~\ref{lem:lls}]
By Lemma~\ref{lem:equality},  there exists a degree distribution \(\vec{d}_+\) on \(X\) with \(\sum_{j=i+1}^g d_+^j = d - a^i_{r-m} = b^i_m\) for which 
\begin{equation}\label{eq:h0_ip1}
h^0(X^{>i}, L_{\vec{d}_+}|_{X^{>i}}) = r+1 - (r-m) = m+1, \qquad \text{and}\qquad \text{\(\ev_{p_i}\)  is surjective.}
\end{equation}
By definition of \(a^{i-1}_n\), for all degree distributions \(\vec{d}\) with total degree \(a^{i-1}_n\) on \(X^{\leq i-1}\) --- in particular those that agree with \(\vec{d}_+\) on \(X^{>i}\) --- we have 
\begin{equation}\label{eq:h0_i}
h^0(X^{>{i-1}}, L_{\vec{d}}|_{X^{>i-1}}) \geq r+1 - n.
\end{equation}
Let \(\vec{d}\) be such a degree distribution with \(\sum_{j=1}^{i-1} d^j = a^{i-1}_n\) and which agrees with \(\vec{d}_+\) when restricted to \(X^{>i}\).  Since this forces \( \sum_{j=i+1}^g d^j = b^i_m\), we conclude that \(d^i = d - a^{i-1}_n - b^i_m\).  Hence, by the procedure for converting between the aspects and a limit with any given degree distribution \eqref{eq:twist}, we have
\[ L_{\vec{d}}|_{E^i} \simeq L^i(-a^{i-1}_n p^{i-1} - b^i_m p^i).\]
In order for \eqref{eq:h0_ip1} and \eqref{eq:h0_i} to both hold, the difference must be made up by sections on \(E^i\), and we have
\[ h^0(E^i, L_{\vec{d}}|_{E^i}) = h^0(E^i, L^i(-a^{i-1}_np^{i-1} - b^i_m p^i)) \geq (r + 1 - n) - (m+1) + 1, \]
where the last \(+1\) is for the agreement condition at \(p^i\).
\end{proof}

In the proof of Brill--Noether Nonexistence, we will apply Lemma~\ref{lem:lls} via the following easy corollary.

\begin{cor}\label{cor:across_comp}
If \(\vec{L} \in LW^r_dX\), then
\(h^0(E^i, L^i(-a^{i-1}_np^{i-1} - b^i_{r-n} p^i)) \geq 1\).
\end{cor}

In words, Corollary~\ref{cor:across_comp} says that the \(i\)th aspect line bundle \(L^i\) has a nonzero section that simultaneously vanishes to order \(a^{i-1}_n\) and \(p^{i-1}\) and to order \(b^i_{r-n}\) at \(p^i\).  Since \(L^i\) has degree \(d\), this immediately implies:

\begin{cor}\label{cor:star}
Let \(\vec{L}\) be an \(r\)-positive line bundle of degree \(d\).  Then, for all \(1 \leq i\leq g\) and all \(0 \leq n \leq r\), we have \(a^{i-1}_n + b^i_{r-n} \leq d\), and if equality holds then
\begin{equation}\label{eq:star}\tag{\(\bigstar^i_n\)}
L^i \simeq \O_{E^i}(a^{i-1}_n p^{i-1} + b^i_{r-n}p^i).
\end{equation}
\end{cor}

The condition~\eqref{eq:star} is important since it \textit{determines} the aspect \(L^i\); the fact that this can only occur once per component (since \(p^{i-1}\) and \(p^i\) are assumed to be general, and so do not satisfy any linear relations) is the key to the proof of Brill--Noether Nonexistence.

\begin{proof}[Proof of Brill--Noether Nonexistence Theorem~\ref{thm:BN_nonexistence}]
Assume that \((g, r, d)\) is such that \(\rho(g, r,d ) =  g - (r+1)(g-d+r) < 0 \).
The semicontinuity theorem implies that it is enough to prove that \(LW^r_dX = \emptyset\).  Suppose to the contrary that \(\vec{L} \in LW^r_dX\) is an \(r\)-positive limit line bundle on \(X\) of degree \(d\).
Let \(a^i_n \colonequals a^i_n(\vec{L})\) and \(b^i_m \colonequals b^i_m(\vec{L})\).  We can pictorially arrange this as below.

\begin{center}
\begin{tikzpicture}
\draw (0, 0) .. controls (1, -1) and (2, -1) .. (3, 0);
\draw (2, 0) .. controls (3, -1) and (4, -1) .. (5, 0);
\draw (4, 0) .. controls (5, -1) and (6, -1) .. (7, 0);
\filldraw (0.5, -0.42) circle[radius=0.03];
\filldraw (2.5, -0.42) circle[radius=0.03];
\filldraw (4.5, -0.42) circle[radius=0.03];
\draw (0.5, -0.75) node{$p^0$};
\draw (2.5, -0.75) node{$p^1$};
\draw (4.5, -0.75) node{$p^2$};
\draw (1.5, -0.4) node{$E^1$};
\draw (3.5, -0.4) node{$E^2$};
\draw (5.5, -0.4) node{$E^3$};
\filldraw (7.8, -0.5) circle[radius=0.02];
\filldraw (7.5, -0.5) circle[radius=0.02];
\filldraw (7.2, -0.5) circle[radius=0.02];
\draw (8, 0) .. controls (9, -1) and (10, -1) .. (11, 0);
\draw (10, 0) .. controls (11, -1) and (12, -1) .. (13, 0);
\filldraw (10.5, -0.42) circle[radius=0.03];
\filldraw (12.5, -0.42) circle[radius=0.03];
\draw (10.5, -0.8) node{$p^{g - 1}$};
\draw (12.5, -0.8) node{$p^g$};
\draw (9.5, -0.4) node{$E^{g - 1}$};
\draw (11.5, -0.4) node{$E^g$};
\draw (0.5, -1.25) node[right]{{\color{violet}\(0\)}};
\draw (0.5, -2) node[right]{{\color{violet}\(1\)}};
\draw (0.5, -2.75) node[right]{{\color{violet}\(\vdots\)}};
\draw (0.2, -3.5) node[right]{{\color{violet}\(r-1\)}};
\draw (0.5, -4.25) node[right]{{\color{violet}\(r\)}};
\draw (2.5, -1.25) node[right]{{\color{violet}\(a^1_0\)}};
\draw (2.5, -2) node[right]{{\color{violet}\(a^1_1\)}};
\draw (2.5, -2.75) node[right]{{\color{violet}\(\vdots\)}};
\draw (2.5, -3.5) node[right]{{\color{violet}\(a^1_{r-1}\)}};
\draw (2.5, -4.25) node[right]{{\color{violet}\(a^1_r\)}};
\draw (4.5, -1.25) node[right]{{\color{violet}\(a^2_0\)}};
\draw (4.5, -2) node[right]{{\color{violet}\(a^2_1\)}};
\draw (4.5, -2.75) node[right]{{\color{violet}\(\vdots\)}};
\draw (4.5, -3.5) node[right]{{\color{violet}\(a^2_{r-1}\)}};
\draw (4.5, -4.25) node[right]{{\color{violet}\(a^2_r\)}};
\draw (2.5, -1.25) node[left]{{\color{blue}\(b^1_0\)}};
\draw (2.5, -2) node[left]{{\color{blue}\(b^1_1\)}};
\draw (2.5, -2.75) node[left]{{\color{blue}\(\vdots\)}};
\draw (2.5, -3.5) node[left]{{\color{blue}\(b^1_{r-1}\)}};
\draw (2.5, -4.25) node[left]{{\color{blue}\(b^1_r\)}};
\draw (4.5, -1.25) node[left]{{\color{blue}\(b^2_0\)}};
\draw (4.5, -2) node[left]{{\color{blue}\(b^2_1\)}};
\draw (4.5, -2.75) node[left]{{\color{blue}\(\vdots\)}};
\draw (4.5, -3.5) node[left]{{\color{blue}\(b^2_{r-1}\)}};
\draw (4.5, -4.25) node[left]{{\color{blue}\(b^2_r\)}};
\draw (10.5, -1.25) node[right]{{\color{violet}\(a^{g-1}_0\)}};
\draw (10.5, -2) node[right]{{\color{violet}\(a^{g-1}_1\)}};
\draw (10.5, -2.75) node[right]{{\color{violet}\(\vdots\)}};
\draw (10.5, -3.5) node[right]{{\color{violet}\(a^{g-1}_{r-1}\)}};
\draw (10.5, -4.25) node[right]{{\color{violet}\(a^{g-1}_r\)}};
\draw (10.5, -1.25) node[left]{{\color{blue}\(b^{g-1}_0\)}};
\draw (10.5, -2) node[left]{{\color{blue}\(b^{g-1}_1\)}};
\draw (10.5, -2.75) node[left]{{\color{blue}\(\vdots\)}};
\draw (10.5, -3.5) node[left]{{\color{blue}\(b^{g-1}_{r-1}\)}};
\draw (10.5, -4.25) node[left]{{\color{blue}\(b^{g-1}_r\)}};
\draw (12.5, -1.25) node[left]{{\color{blue}\(0\)}};
\draw (12.5, -2) node[left]{{\color{blue}\(1\)}};
\draw (12.5, -2.75) node[left]{{\color{blue}\(\vdots\)}};
\draw (12.8, -3.5) node[left]{{\color{blue}\({r-1}\)}};
\draw (12.5, -4.25) node[left]{{\color{blue}\(r\)}};
\draw [decorate,
    decoration = {brace,mirror}, thick] (9.5,-5) --  (11.5,-5);
\draw (10.5, -5) node[below]{{\footnotesize across a node,}};
\draw (10.5, -5.3) node[below]{{\footnotesize \(a^i_n + b^i_{r-n} = d\)}};
\draw [decorate,
    decoration = {brace,mirror}, thick] (2.5,-5) --  (4.5,-5);
\draw (3.5, -5) node[below]{{\footnotesize across a component,}};
\draw (3.5, -5.3) node[below]{{\footnotesize \(a^{i-1}_n + b^i_{r-n} \leq d\)}};
\end{tikzpicture}
\end{center}

For fixed \(n\) between \(0\) and \(r\), considering the sum of complementary terms of the vanishing sequences at \(p^{i-1}\) and at \(p^i\) over all components, Corollary~\ref{cor:star} and \eqref{eq:b_complement} imply:
\begin{align*}
g(d-1) + \#\{i  : \eqref{eq:star} \text{ holds}\} &\geq \sum_{i=1}^g a^{i-1}_n + b^i_{r-n} \\
&= n + \left( \sum_{i=1}^{g-1} b^i_{r-n} + a^i_n\right) + (r- n)\\
&= r + (g-1)d.
\end{align*}
In other words, for fixed \(n\), the number of components on which \eqref{eq:star} holds --- and the aspect \(L^i\) is determined --- is at least \(g - d + r\).  Thus, considering all \(0 \leq n \leq r\), we see that this determines \((r+1)(g-d+r)\) of the aspects (again, because there cannot be any overlaps since \(p^{i-1}\) and \(p^i\) satisfy no nontrivial linear relations on \(E^i\)).  Since there are only \(g\) aspects in total, we see that \(g \geq (r+1)(g-d+r)\), in other words \(\rho(g, r, d) \geq 0\).
\end{proof}

\subsection{Embedded degeneration and deformation theory}\label{sec:BN_exist}
In this setting we begin not with an abstract curve \(C\), but rather with a curve \(C \hookrightarrow \pp^r\) already equipped with a projective embedding.  We then make use of degenerations of \(C\) in the projective space \(\pp^r\).
To illustrate this technique we will prove the following version of the Brill--Noether Existence Theorem.

\begin{thm}[Brill--Noether Existence]\label{thm:BN_existence}

Suppose that \(\rho(g, r, d) = g - (r+1)(g-d+r) \geq 0\) and \(r \geq 3\).  Then there exists a component \(\M_g(\pp^r, d)^{\BN}\) of \(\M_g(\pp^r, d)\) such that the moduli of the source map
\begin{equation}\label{eq:moduli_source}
\M_g(\pp^r, d)^{\BN} \to \M_g
\end{equation}
is generically smooth of the expected relative dimension \(\rho(g, r, d) + (r+1)^2 - 1\), and, furthermore, the general map parameterized by  \(\M_g(\pp^r, d)^{\BN}\) is a nondegenerate embedding.
\end{thm}

Since all of the properties guaranteed in this theorem are open, for each tuple \((g, r, d)\) such that \(\rho(g, r, d) \geq 0\), it suffices to show that there exists a nondegenerate embedding \(f \colon C \hookrightarrow \pp^r\) of degree \(d\) from a smooth curve \(C\) of genus \(g\) at which that map~\eqref{eq:moduli_source} is smooth of relative dimension \(\rho(g, r, d) + (r+1)^2 - 1\).  By deformation theory, the relative tangent space of the map~\eqref{eq:moduli_source} at the point \([f \colon C \to \pp^r]\) is isomorphic to the vector space \(H^0(C, f^*T_{\pp^r})\).  Furthermore, obstructions to smoothness of the map~\eqref{eq:moduli_source} at the point \([f \colon C \to \pp^r]\) lie in the vector space \(H^1(C, f^*T_{\pp^r})\).  It suffices to show, therefore, that 
\(h^0(C, f^*T_{\pp^r}) = \rho(g, r, d) + (r+1)^2 - 1\) and  
\(h^1(C, f^*T_{\pp^r}) = 0\).  In fact, these two properties are equivalent since
\[\chi(C, f^*T_{\pp^r}) \colonequals h^0(C, f^*T_{\pp^r}) - h^1(C, f^*T_{\pp^r})\]
can be computed by the Euler sequence
\begin{equation}\label{eq:euler}
0 \to \O_{\pp^r} \to \O_{\pp^r}(1)^{\oplus r+1} \to T_{\pp^r} \to 0
\end{equation}
 to be \(\rho(g, r, d) + (r+1)^2 - 1\).  Our goal, therefore, is to show that for each tuple \((g, r, d)\) such that \(\rho(g, r, d) \geq 0\), there exists a point \([f \colon C \to \pp^r]\) of \(\M_g(\pp^r, d)\) such that \(f\) is a nondegenerate embedding and \(h^1(C, f^*T_{\pp^r}) = 0\).

The previous paragraph covers the ``deformation theory'' part of the title of this section.  This technique is particularly powerful when combined with embedded degeneration and induction.  That is, we compactify the space \(\M_g(\pp^r, d)\) to \(\bar{\M}_g(\pp^r, d)\), the space of stable maps from a nodal source curve of genus \(g\).  By the semicontinuity theorem and openness of the property of being a nondegenerate embedding, if there is any point \([f_0 \colon C_0 \to \pp^r] \in \bar{\M}_g(\pp^r, d)\) for which \(f_0\) is a nondegenerate embedding with \(h^1(C_0, f_0^*T_{\pp^r}) = 0\), then the same is true for a general point in that component of \(\bar{\M}_g(\pp^r, d)\) and hence a general point in the interior \(\M_g(\pp^r, d)\), as desired.

We must prove a result for each tuple of nonnegative integers \((g, r, d)\) with \(r\geq 3\) and such that \(g - (r+1)(g - d +r) \geq 0\).  In this proof we will think of \(r\) as fixed and view the Brill--Noether inequality as prescribing how large \(d\) must be for each \(g\). Rearranging terms, \(\rho(g, r, d) \geq 0\) is equivalent to
\begin{equation}\label{eq:BN_rearrange}
d \geq \left\lceil\left(\frac{r}{r+1}\right) g\right\rceil + r.
\end{equation}
The collection of pairs \((d, g)\) satisfying~\eqref{eq:BN_rearrange} for \(r=3\) is pictured in Figure~\ref{fig:BN_lattice}.  
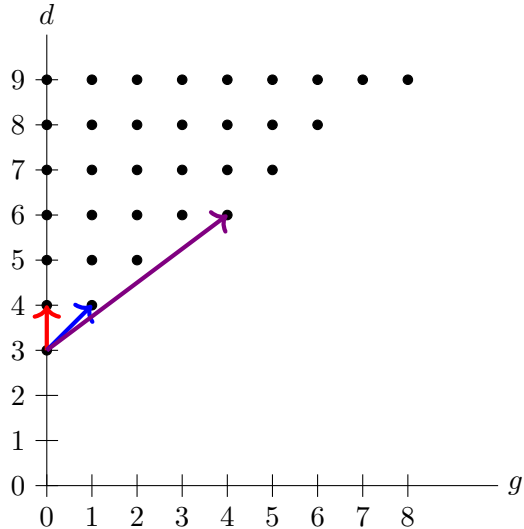
\begin{figure}[h!]
\begin{tikzpicture}[scale=.6]

\draw (0,0) -- (10, 0);
\draw (0,0) -- (0, 10);
\draw (10,0) node[right]{\(g\)};
\draw (0, 10) node [above]{\(d\)};

\filldraw (0, 3) circle (3pt);
\filldraw (0, 4) circle (3pt);
\filldraw (0, 5) circle (3pt);
\filldraw (0, 6) circle (3pt);
\filldraw (0, 7) circle (3pt);
\filldraw (0, 8) circle (3pt);
\filldraw (0, 9) circle (3pt);
\filldraw (1, 4) circle (3pt);
\filldraw (1, 5) circle (3pt);
\filldraw (1, 6) circle (3pt);
\filldraw (1, 7) circle (3pt);
\filldraw (1, 8) circle (3pt);
\filldraw (1, 9) circle (3pt);
\filldraw (2, 5) circle (3pt);
\filldraw (2, 6) circle (3pt);
\filldraw (2, 7) circle (3pt);
\filldraw (2, 8) circle (3pt);
\filldraw (2, 9) circle (3pt);
\filldraw (3, 6) circle (3pt);
\filldraw (3, 7) circle (3pt);
\filldraw (3, 8) circle (3pt);
\filldraw (3, 9) circle (3pt);
\filldraw (4, 6) circle (3pt);
\filldraw (4, 7) circle (3pt);
\filldraw (4, 8) circle (3pt);
\filldraw (4, 9) circle (3pt);
\filldraw (5, 7) circle (3pt);
\filldraw (5, 8) circle (3pt);
\filldraw (5, 9) circle (3pt);
\filldraw (6, 8) circle (3pt);
\filldraw (6, 9) circle (3pt);
\filldraw (7, 9) circle (3pt);
\filldraw (8, 9) circle (3pt);

\draw (0, -.25)--(0, .25);
\draw (0, -.25) node[below]{\(0\)};
\draw (1, -.25)--(1, .25);
\draw (1, -.25) node[below]{\(1\)};
\draw (2, -.25)--(2, .25);
\draw (2, -.25) node[below]{\(2\)};
\draw (3, -.25)--(3, .25);
\draw (3, -.25) node[below]{\(3\)};
\draw (4, -.25)--(4, .25);
\draw (4, -.25) node[below]{\(4\)};
\draw (5, -.25)--(5, .25);
\draw (5, -.25) node[below]{\(5\)};
\draw (6, -.25)--(6, .25);
\draw (6, -.25) node[below]{\(6\)};
\draw (7, -.25)--(7, .25);
\draw (7, -.25) node[below]{\(7\)};
\draw (8, -.25)--(8, .25);
\draw (8, -.25) node[below]{\(8\)};

\draw (-.25, 0) -- (.25, 0);
\draw (-.25, 0) node[left]{\(0\)};
\draw (-.25, 1) -- (.25, 1);
\draw (-.25, 1) node[left]{\(1\)};
\draw (-.25, 2) -- (.25, 2);
\draw (-.25, 2) node[left]{\(2\)};
\draw (-.25, 3) -- (.25, 3);
\draw (-.25, 3) node[left]{\(3\)};
\draw (-.25, 4) -- (.25, 4);
\draw (-.25, 4) node[left]{\(4\)};
\draw (-.25, 5) -- (.25, 5);
\draw (-.25, 5) node[left]{\(5\)};
\draw (-.25, 6) -- (.25, 6);
\draw (-.25, 6) node[left]{\(6\)};
\draw (-.25, 7) -- (.25, 7);
\draw (-.25, 7) node[left]{\(7\)};
\draw (-.25, 8) -- (.25, 8);
\draw (-.25, 8) node[left]{\(8\)};
\draw (-.25, 9) -- (.25, 9);
\draw (-.25, 9) node[left]{\(9\)};

\draw[->, color=red, ultra thick] (0, 3) -- (0,4);
\draw[->, color=blue, ultra thick] (0, 3) -- (1,4);
\draw[->, color=violet, ultra thick] (0, 3) -- (4,6);

\end{tikzpicture}
\caption{For fixed \(r\) (here \(r=3\)) the set of tuples \((g, d)\) such that \(\rho(g, r, d) \geq 0\) (pictured here in black) can be obtained from the rational normal curve \((0, r)\) by successively applying one of   {\color{red}(A) \((d, g) \mapsto (d+1, g)\)}, {\color{blue}(B) \((d, g) \mapsto (d+1, g+1)\)}, or finally {\color{violet}(C) \((d, g) \mapsto (d+r, g+r+1)\)}.}\label{fig:BN_lattice}
\end{figure}

When \(g=0\), the smallest \(d\) satisfying~\eqref{eq:BN_rearrange} is \(d=r\), yielding the invariants of a rational normal curve \((g, r, d) = (0, r, r)\).  The other interesting case is \(g = r+1\), in which case~\eqref{eq:BN_rearrange} becomes \(d \geq 2r\); the invariants \((g, r, d) = (r+1, r, 2r)\) describe canonically embedded curves.  In fact, the set of tuples \((g, d)\) such that \(\rho(g, r, d) \geq 0\) can be obtained from the rational normal curve \((0, r)\) by successively applying one of three basic moves:
\begin{enumerate}
\item[(A)] Increase \(d\) by one: \((d, g) \mapsto (d+1, g)\), 
\item[(B)] Increase both \(d\) and \(g\) by one: \((d, g) \mapsto (d+1, g+1)\), 
\item [(C)] Increase \(d\) by \(r\) and \(g\) by \(r+1\): \((d, g) \mapsto (d+r, g+r+1)\).
\end{enumerate}
The last of these goes from the rational normal curve to a canonical curve.

This numerical property of the Brill--Noether number naturally suggests an inductive argument with three inductive steps (A), (B), and (C) as above.  Moreover, we can realize each of these by building a reducible nodal curve with the new invariants out of one with the old invariants: if \(C\) has degree \(d\) and genus \(g\), then 
\begin{enumerate}
\item[(A)] Attach to \(C\) a quasitransverse \(1\)-secant line.
\item[(B)] Attach to \(C\) a quasitransverse \(2\)-secant line.
\item [(C)] Attach to \(C\) a quasitransverse \((r+2)\)-secant rational normal curve.
\end{enumerate}

The only remaining ingredient is the following elementary lemma determining the isomorphism class of the restricted tangent bundle of a line and a rational normal curve.

\begin{lem}\label{lem:line_rnc}
\begin{enumerate}
\item Let \(R \subset \pp^r\) be a line.  Then \(T_{\pp^r}|_R \simeq \O_{\pp^1}(1)^{\oplus r-1} \oplus \O_{\pp^1}(2)\).
\item Let \(R \subset \pp^r\) be a rational normal curve.  Then \(T_{\pp^r}|_R \simeq \O_{\pp^1}(r+1)^{\oplus r}\).
\end{enumerate}
\end{lem}
\begin{proof}
\begin{enumerate}
\item This follows from the normal bundle exact sequence
\[0 \to T_R \simeq \O_{\pp^1}(2) \to T_{\pp^r}|_R \to N_{R/\pp^r} \simeq \O_{\pp^1}(1)^{\oplus r-1} \to 0,\]
since the extension is necessarily split.
\item Consider the Euler exact sequence \eqref{eq:euler} restricted to \(R\):
\[0 \to \O_{\pp^1} \to \O_{\pp^1}(r)^{\oplus r+1} \to T_{\pp^r}|_R \to 0.\]
The first map \(\O_{\pp^1} \to \O_{\pp^1}(r)^{\oplus r+1}\) is multiplication by the (linearly independent) sections of \(\O_{\pp^1}(r)\) defining the nondegenerate embedding \(R \hookrightarrow \pp^r\).
The bundle \(T_{\pp^r}|_R\) has rank \(r\) and degree \(r(r+1)\).  To show that it is perfectly balanced, it therefore suffices to rule out a summand of degree at most \(r\); equivalently, a nonzero map \(T_{\pp^r}|_R \to \O_{\pp^1}(r)\).  Again considering the Euler sequence, any such map yields a nonzero map \(\O_{\pp^1}(r)^{\oplus r+1} \to \O_{\pp^1}(r)\) that vanishes when precomposed with the map \(\O_{\pp^1} \to \O_{\pp^1}(r)^{\oplus r+1}\), which is impossible since the sections defining the  map \(\O_{\pp^1} \to \O_{\pp^1}(r)^{\oplus r+1}\) are linearly independent.\qedhere
\end{enumerate}
\end{proof}

\begin{proof}[Proof of Theorem~\ref{thm:BN_existence}]
As explained above, deformation theory reduces this to finding, for each tuple \((g, r, d)\) such that \(\rho(g, r, d) \geq 0\), a nondegenerate embedding of a nodal curve \(f_0 \colon C_0 \to \pp^r\) such that \(h^1(C_0, f_0^*T_{\pp^r}) = 0\).
For each fixed \(r \geq 3\), we will prove this by induction on \((g, d)\).  The base case is \((g, d) = (0,r)\) rational normal curves, for which the result follows from Lemma~\ref{lem:line_rnc}.

Now assume by induction that there exists a nondegenerate degree \(d\) embedding \([f \colon C \to \pp^r]\) of a smooth curve \(C\) of genus \(g\) for which \(h^1(C, f^*T_{\pp^r})=0\).  For each of the inductive moves (A), (B), and (C), let \(R\) denote the rational curve that is attached, so \(C_0 \colonequals C \cup R\) has the new desired invariants.  Restriction to \(C\) yields an exact sequence
\[0 \to T_{\pp^r}|_R(-C \cap R) \to T_{\pp^r}|_{C \cup R} \to T_{\pp^r}|_C \to 0.\]
We know that \(h^1(C, f^*T_{\pp^r})=0\) by induction.  Furthermore, by Lemma~\ref{lem:line_rnc} and the fact that \(\# C \cap R = 1\), \(2\) and \(r+2\) respectively, in each of our inductive moves we have
\begin{enumerate}
\item[(A)] \( T_{\pp^r}|_R(-C \cap R) \simeq \O^{\oplus r-1} \oplus \O(1)\),
\item[(B)] \( T_{\pp^r}|_R(-C \cap R) \simeq \O(-1)^{\oplus r-1} \oplus \O\),
\item [(C)] \( T_{\pp^r}|_R(-C \cap R) \simeq \O(-1)^{\oplus r}\).
\end{enumerate}
In each of these cases, the bundle has no higher cohomology.
Hence, considering the long exact sequence in cohomology, we have \(h^1(C \cup R, T_{\pp^r}|_{C \cup R}) = 0\).  Thus \(f_0 \colon C \cup R \to \pp^r\) satisfies the conclusions for the new invariants.
\end{proof}

\subsection{Modifications of normal bundles}\label{sec:normal}
The final technique we will illustrate also concerns embedded degeneration of Brill--Noether curves, but we will focus on the normal bundle.  We will illustrate new subtleties that arise in this case by proving the following very special case of Theorem~\ref{thm:interpolation}, which is originally due to Sacchiero \cite[(II) on page 1]{sacc_int}.  The proof we give is not the quickest way to a solution, but will illustrate some of the powerful techniques that have been developed in \cite{aly, interpolation, grass, canonical} in recent years.

\begin{prop}\label{prop:odd_deg}
Let \(C \subset \pp^3\) be a general curve of genus \(0\) and odd degree \(d\geq 3\).  Then
\[N_C \simeq \O_{\pp^1}(2d-1) \oplus \O_{\pp^1}(2d-1)\]
is perfectly balanced.
In particular, rational curves of odd degree in \(\pp^3\) interpolate the expected number of general points.
\end{prop}

\begin{warning}
We do not restrict to odd degree rational curves for trivial reasons: the result is false as stated for rational curves of even degree!  The reason is that it fails in characteristic \(2\).  More generally, if  \(d \not\equiv 1 \pmod{r-1}\), then the normal bundle of a rational curve of degree \(d\) in \(\pp^r\) is never balanced.  See \cite[Section 2.2]{interpolation} for more details on this phenomena.
\end{warning}

Balancedness of \(N_C\) for \(C \subset \pp^3\) a rational curve of degree \(d\) is equivalent to \(h^0(C, N_C(-2d)) = 0\), equivalently \(h^1(C, N_C(-2d)) = 0\).
The property of interpolation (\cite[Definition 1.3]{interpolation}) for \(N_C\) is equivalent to  \(N_C\) being balanced since it is nonspecial (for example, by the arguments of the previous Section~\ref{sec:BN_exist}, since \(N_C\) is a quotient of \(T_{\pp^r}|_C\)).  

A first technique is to project from a point \(q \in C\) and use the resulting exact sequence~\eqref{eq:proj} below.

\begin{defin}
Let \(C \subset \pp^r\) be a smooth curve, and let \(q \in \pp^r\).  Write \(S\) for the cone over \(C\) with vertex \(q\).
The \defi{(normal) pointing bundle} \(N_{C \to q}\) is the unique line subbundle of \(N_C\) agreeing with \(N_{C/S}\) away from \(q\) (which is only relevant if \(q \in C\)).
\end{defin}

 The pointing bundle is so called because its sections are normal directions that ``point towards \(q\)".
A local calculation \cite[Propositions 6.2 and 6.3]{aly} shows that 
\begin{equation}\label{eq:pointing}
N_{C \to q} \simeq \begin{cases} \O_C(1) & :  \ q \not\in T_pC \text{ for any \(p \in C\)}, \\  
\O_C(1)(2q) & :  \ q \in C^{\text{sm}} \text{ is not an inflection point and \(q \not\in T_pC\) for any \(p \neq q\).} \end{cases}
\end{equation}

Assume that the point \(q \in C^{\text{sm}}\) is a general point.  Projection from \(q\) yields a map \(\pi_q \colon C \to \pp^{r-1}\) and an exact sequence
\begin{equation}\label{eq:proj}
0 \to N_{C \to q} \simeq \O_C(1)(2q) \to N_C \to N_{\pi_q}(q) \to 0,
\end{equation}
where \(N_{\pi_q}\) is the normal sheaf of the map \(\pi_q\) (which agrees with the normal bundle \(N_{\pi_q(C)}\) when \(\pi_q\) is an embedding).  In the case where \(r=3\) and \(\pi_q\) is usually not an embedding, the normal sheaf can be computed via adjuction \(N_{\pi_q} \simeq \pi_q^*K_{\pp^2}^\vee \otimes K_{C}\).  In the case where \(C \subset \pp^3\) is a rational curve of degree \(d\) this yields \(N_{\pi_q} \simeq \O_{\pp^1}(3(d-1)) \otimes \O_{\pp^1}(-2) \simeq \O_{\pp^1}(3d - 5)\). 

This technique alone suffices in the following first case.

\begin{example}[Twisted cubic curves]
Let \(C \subset \pp^3\) be a twisted cubic curve.  For \(q\in C\), the image \(\pi_q(C)\) is a smooth plane conic and hence has normal bundle isomorphic to \(\O_{\pp^1}(4)\).  Using this and \eqref{eq:pointing}, the projection exact sequence~\eqref{eq:proj} therefore simplifies to
\[0 \to \O_{\pp^1}(5) \to N_C \to \O_{\pp^1}(5) \to 0,\]
and therefore \(N_C \simeq \O_{\pp^1}(5)^{\oplus 2}\). \hfill\(\righthalfcup\)
\end{example}

What made the projection exact sequence immediately effective was that it was perfectly balanced: the degree of the sub line bundle and the quotient line bundle were equal.  In general, the projection exact sequence is far from balanced: the subbundle \(\O_C(1)(2	q)\) has degree \(d+2\) whereas the quotient \(N_{\pi_q}(q)\) has degree \(3d - 4\).

In order to treat degrees larger than \(3\), as in the previous Section~\ref{sec:BN_exist}, we will degenerate \(C\) to a reducible nodal curve \(X \cup Y\).  By the semicontinuity theorem, it suffices to prove the necessary cohomological vanishing that is equivalent to balancedness/interpolation after degeneration. 

Degenerative arguments involving the normal bundle are substantially more difficult than those using the restricted tangent bundle, such as in the previous Section~\ref{sec:BN_exist}.  The essential difference is that \(T_{\pp^r}\) is a vector bundle on all of \(\pp^r\), while the normal bundle \(N_C\) is only defined on the curve \(C \subset \pp^r\).  In the first case, when we degenerate \(C\) to a reducible nodal curve \(X \cup Y\), it is clear that \(T_{\pp^r}|_C\) degenerates to the vector bundle \(T_{\pp^r}|_X\) on \(X\) and \(T_{\pp^r}|_Y\) on \(Y\) with the obvious gluing at the nodes.  On the other hand, although \(N_{X \cup Y}\) is a vector bundle on the nodal curve \(X \cup Y\), 
its restriction to a component \(N_{X \cup Y}|_X\) is not isomorphic to \(N_X\).  This fact is not surprising: deformation theory guarantees that sections of the normal bundle parameterize first order deformations of the curve, and deformations smoothing the node at a point \(p \in X \cap Y\) are necessarily nonregular at \(p\) when viewed in \(N_X\).

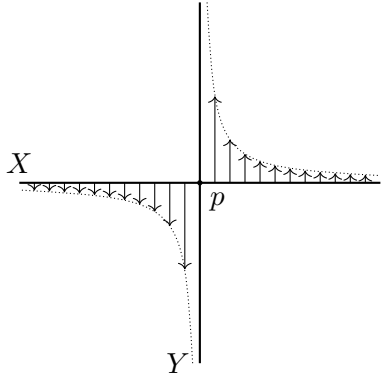
\begin{center}
\begin{tikzpicture}[scale=.8]
\filldraw (0,0) circle (1pt);
\draw (0,0) node[below right]{\(p\)};
\draw (-3, 0) node[above]{{\color{black}\(X\)}};
\draw (0, -3) node[left]{{\color{black}\(Y\)}};
\draw[thick, color=black] (-3, 0) -- (3, 0);
\draw[thick, color=black] (0, -3) -- (0, 3);
\draw[densely dotted] (0.12, 3) .. controls (0.28, 0.28) .. (3, 0.12);
\draw[densely dotted] (-0.12, -3) .. controls (-0.28, -0.28) .. (-3, -0.12);
\draw[->] (2.75, 0) -- (2.75, 0.1309090909090909);
\draw[->] (2.5, 0) -- (2.5, 0.144);
\draw[->] (2.25, 0) -- (2.25, 0.16);
\draw[->] (2, 0) -- (2, 0.18);
\draw[->] (1.75, 0) -- (1.75, 0.20571428571428574);
\draw[->] (1.5, 0) -- (1.5, 0.24);
\draw[->] (1.25, 0) -- (1.25, 0.288);
\draw[->] (1, 0) -- (1, 0.36);
\draw[->] (0.75, 0) -- (0.75, 0.48);
\draw[->] (0.5, 0) -- (0.5, 0.72);
\draw[->] (0.25, 0) -- (0.25, 1.44);
\draw[->] (-2.75, 0) -- (-2.75, -0.1309090909090909);
\draw[->] (-2.5, 0) -- (-2.5, -0.144);
\draw[->] (-2.25, 0) -- (-2.25, -0.16);
\draw[->] (-2, 0) -- (-2, -0.18);
\draw[->] (-1.75, 0) -- (-1.75, -0.20571428571428574);
\draw[->] (-1.5, 0) -- (-1.5, -0.24);
\draw[->] (-1.25, 0) -- (-1.25, -0.288);
\draw[->] (-1, 0) -- (-1, -0.36);
\draw[->] (-0.75, 0) -- (-0.75, -0.48);
\draw[->] (-0.5, 0) -- (-0.5, -0.72);
\draw[->] (-0.25, 0) -- (-0.25, -1.44);
\end{tikzpicture}
\end{center}

We have a good description of \(N_{X \cup Y}|_X\) in terms of \(N_X\) and likewise for \(Y\) (see Lemma~\ref{lem:HH} below) along these lines.  However, the two restrictions \(N_{X \cup Y}|_X\) and \(N_{X\cup Y}|_Y\) do not alone suffice to describe the bundle \(N_{X \cup Y}\) and the gluing conditions at the nodes are very subtle, and only well understood in specific degenerations.

 The first task is to understand how the normal bundle of a reducible nodal curve \(N_{X \cup Y}\) relates to the normal bundle \(N_X\) of a component.  This is a classical result of Hartshorne--Hirschowitz.  To state this result in full generality we need the following definition.

\begin{defin}
Let \(E\) be a vector bundle on a variety \(X\), let \(D \subset X\) be a Cartier divisor, and let \(F \subset E|_D\) be a subbundle.  The \defi{negative elementary modification of \(E\) along \(D\) towards \(F\)} is the vector bundle
\[E[D\negmod F] \colonequals \ker\left( E \to E|_D/F\right).\]
The \defi{positive elementary modification of \(E\) along \(D\) towards \(F\)} is \(E[D \posmod F] \colonequals E[D \negmod F](D)\).
\end{defin}

The bundle \(E[D\posmod F]\) is more positive than \(E\) insofar as its sections are allowed to have poles along \(D\) in the subbundle \(F\).   For example, if \(E \simeq F \oplus F'\), then \(E[D \posmod F|_D] \simeq F(D) \oplus F'\).  More generally, given an exact sequence \(0 \to F \to E \to Q \to 0\), then induced exact sequence for the modification is
\begin{equation}\label{eq:mod_ses}
0 \to F(D) \to E[D \posmod F|_D] \to Q \to 0.
\end{equation}
On the other hand, if we have an exact sequence \(0 \to F \to E|_D \to Q \to 0\) on \(D\), then there is an induced exact sequence
\begin{equation}\label{eq:mod_res}
0 \to Q \to E[D \posmod F]|_D \to F \otimes \O(D)|_D \to 0.
\end{equation}

\begin{lem}[{Hartshorne--Hirschowitz \cite[Corollary 3.2]{HH}}]\label{lem:HH}
Let \(X \cup Y\) be a reducible nodal curve with \(X \cap Y = \{p_1, \dots, p_m\}\) and for each \(i\), let \(q_i \in T_{p_i}Y \smallsetminus p_i\). Then 
\[N_{X \cup Y}|_X \simeq N_X[p_1 \posmod N_{C \to q_1}] \cdots [p_m \posmod N_{C \to q_m}].\]
\end{lem}

\begin{proof}[Proof sketch]
First, assume that \(X \cup Y\) is contained in a smooth surface \(S\).  Then \(X + Y\) is a Cartier divisor on \(S\) and \(N_{X \cup Y} \simeq \O_S(X + Y)|_{X \cup Y}\) and \(N_X \simeq \O_S(X)|_X\).  This implies the relation
\[N_{X \cup Y}|_X \simeq \O_X(X)(X \cap Y) \simeq N_X(X \cap Y),\]
which is the desired result, since for any point \(p \in X\cap Y\), the space \(N_X|_p\) is \(1\)-dimensional and hence isomorphic to \(T_pY\).

The general case follows since this is a local property and a nodal singular is planar, so \(X \cup Y\) is contained in a smooth surface locally in a neighborhood of any point in \(X \cap Y\).
\end{proof}

\begin{example}\label{ex:1sec}
Consider the degeneration (A) from the previous Section~\ref{sec:BN_exist}, that is, the union \(X \cup_p L\) of a curve \(X\) and a quasitransverse \(1\)-secant line \(L\) meeting \(X\) at the point \(p\).  Let \(q \in L\smallsetminus p\) and let \(S\) be the cone over \(X\) with vertex \(q\) (which is smooth along \(X\)).  Let \(\Lambda \colonequals \overline{T_pX, q}\) be the \(2\)-plane containing \(L\) and the tangent line \(T_pX\), which is the cone over \(L\) with vertex a point on the tangent line \(T_pX \smallsetminus p\).  By Lemma~\ref{lem:HH} we know that
\begin{align*}
N_{X \cup L}|_X &\simeq N_X[p \posmod N_{X/S}]\\
N_{X \cup L}|_L & \simeq N_L[p \posmod N_{L/\Lambda}] \simeq \O_{\pp^1}(2) \oplus \O_{\pp^1}(1)^{\oplus r-2}.
\end{align*}
\hfill \(\righthalfcup\)
\end{example}

We now turn to the problem of understanding the gluing data at the points of \(X \cap Y\).  We focus on
 the \(1\)-secant degeneration \(Y = L\) from Example~\ref{ex:1sec}.  This is one of the few cases where a nice description exists, and it is the only degeneration we will need to prove Proposition~\ref{prop:odd_deg}.  The bundle \(N_{X \cup L}|_L \) has distinguished summand \(\O_{\pp^1}(2)\) and the problem is to determine the \(1\)-dimensional subspace that this glues to in \(N_X[p \posmod N_{X/S}]|_p\).  The key observation in \cite[Lemma 8.4]{aly} is that the cone \(S\) is everywhere tangent to the plane \(\Lambda\) along \(L\), in particular at \(p\).  
 
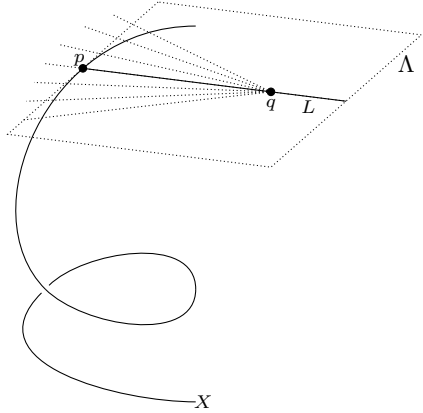
 \begin{center}
 \begin{tikzpicture}
\draw (2, 5) .. controls (0, 5) and (-1, 2.5) .. (0, 1.5);
\draw (0, 1.5) .. controls (0.5, 1) and (2, 0.75) .. (2, 1.5);
\draw (0.05, 1.55) .. controls (0.5, 2) and (2, 2.25) .. (2, 1.5);
\draw (-0.05, 1.45) .. controls (-1, 0.5) and (1, 0) .. (2, 0);
\filldraw (3, 4.125) circle[radius=0.05];
\draw[densely dotted] (0, 4.5) -- (4, 4);
\draw (0.5, 4.4375) -- (4, 4);
\draw[densely dotted] (3, 4.125) -- (-0.3, 3.75);
\draw[densely dotted] (3, 4.125) -- (-0.25, 4);
\draw[densely dotted] (3, 4.125) -- (-0.15, 4.25);
\draw[densely dotted] (3, 4.125) -- (0.2, 4.75);
\draw[densely dotted] (3, 4.125) -- (0.5, 5);
\draw[densely dotted] (3, 4.125) -- (0.9, 5.15);
\draw (2.1, 0) node{{\tiny \(X\)}};
\draw (3.5, 3.93) node{{\tiny \(L\)}};
\draw (3, 3.95) node{{\tiny \(q\)}};
\draw (0.45, 4.55) node{{\tiny \(p\)}};
\filldraw (0.5, 4.4375) circle[radius=0.05];
\draw[densely dotted] (1.5, 5.32) -- (-0.5, 3.56) -- (3, 3.1225) -- (5, 4.8825) -- (1.5, 5.32);
\draw (4.8, 4.5) node{\footnotesize \(\Lambda\)};
\end{tikzpicture}
\end{center}

 Thus the two \(1\)-dimensional subspaces \(\O_{\pp^1}(2) \simeq N_{L/\Lambda}(p)|_p\) and \(N_{L/S}(p)|_p\) agree in \(N_{X\cup L}|_p\).  Since \(N_{L/S}(p)|_p\) glues to \(N_{X/S}(p)|_p\) --- they are both identified with the fiber at \(p\) of the bundle \(N_{X \cup L/S}\) --- this identifies the distinguished \(1\)-dimensional subspace.  This leads to an important inductive tool, which we state in full generality since the proof is no harder than for rational curves in \(\pp^3\).
 
 \begin{lem}[{\cite[Proposition 4.2]{grass}}]\label{lem:1sec}
Let \(X \subset \pp^r\) be a smooth curve, and let \(L = \overline{p,q}\) be a quasitransverse \(1\)-secant line.  If the bundle \(N_X[2p \posmod N_{X \to q}]\) satisfies interpolation, then so does the normal bundle of a general smoothing of \(X \cup L\).
 \end{lem}
 \begin{proof}[Proof sketch]\
 We retain the notation of \(S\) as the cone over \(X\) with vertex \(q\).
 Consider a \(1\)-parameter smoothing \(\C \to \Spec K \llbracket t \rrbracket\) with general fiber \(C\) and special fiber \(X \cup L\).  Let \(\N\) be the bundle on \(\C\) whose restriction to \(C\) is the normal bundle \(N_C\) and whose restriction to \(X \cup L\) is the bundle \(N_{X \cup L}\).  Define 
 \(\N' \colonequals \N[L \posmod N_{L/\Lambda}(p)](L)\).
Using the exact sequence in \eqref{eq:mod_res} with \(Q = \O_{\pp^1}(1)^{\oplus r-2}\) and \(F = N_{L/\Lambda}(p) \simeq \O_{\pp^1}(2)\), we see that \(\N'|_L \simeq \left( \O_{\pp^1}(1)^{\oplus{r-1}}\right)(-p) \simeq \O_L^{\oplus r-1}\).  For the restriction to \(X\) we have
 \begin{align*}
 \N'|_X & \simeq N_{X \cup L}|_X[p \posmod N_{X/S}(p)](p),\\
 & \simeq N_{X}[p \posmod N_{X/S}][p \posmod N_{X/S}(p)](p),\\
 & \simeq N_{X}[2p \posmod N_{X/S}](p).
 \end{align*}
Since \(L\) is a \((-1)\)-curve on \(\C\), it can be contracted to yield \(\C^- \to \Spec K \llbracket t \rrbracket\), and the Theorem on Formal Functions together with the fact that a trivial bundle on \(L\) has no higher cohomology implies that \(\N'\) is pulled back from \(\C^-\).
Thus the bundle \(N_{X}[2p \posmod N_{X/S}](p)\) is a specialization of \(N_C\).  Hence, if
 \(N_{X}[2p \posmod N_{X/S}]\)  satisfies interpolation, then so does \(N_{X}[2p \posmod N_{X/S}](p)\), and consequently so does \(N_C\). 
 \end{proof}
 
 \begin{rem}
 In the statement of Lemma~\ref{lem:1sec}, we could replace \(N_X\) with some modification \(N_X'\), see \cite[Section 4]{grass} for this level of generality. This implies that we may peel off multiple \(1\)-secant lines and apply Lemma~\ref{lem:1sec} repeatedly.
\end{rem}

\begin{warning}
As the proof of Lemma~\ref{lem:1sec} shows, the bundle \(N_{X}[2p \posmod N_{L/S}](p)\) is actually the one that fits into a flat family with \(N_C\) for a general smoothing \(C\) of \(X \cup L\).  In dropping the twist by \(p\) we are reducing the degree by the rank every time we apply this degeneration.  While this does not affect the property of being perfectly balanced, it does change which balanced bundle it is isomorphic to.
\end{warning}
 
Repeatedly applying Lemma~\ref{lem:1sec} can reduce interpolation for \(N_C\) to interpolation for the multiply modified \(N_X[2p_1 \posmod N_{X \to q_1}] \cdots [2p_m \posmod N_{X \to q_m}]\), where \(X\) is a curve of  degree \(d-m\).  This is slightly problematic as an inductive strategy, however, since the modifications accrue.   The miraculous solution is that the problems with the two techniques that we have identified (1) projection sequences are often far from balanced, and (2) degeneration produces accruing modifications, are actually solutions to each other.  We will illustrate this idea in the following proof.

 \begin{proof}[Proof of Proposition~\ref{prop:odd_deg}]
 We will apply the \(1\)-secant degeneration \((d-3)/2\) times to reduce to showing that 
 \[N_{X}[2p_1 \posmod N_{X \to q_1}] \cdots [2p_{(d-3)/2} \posmod N_{X \to q_{(d-3)/2}}]\]
 is perfectly balanced, where \(X\) is a curve of degree \((d+3)/2\).  (Notice: we are crucially using that \(d\) is odd here!)  Now specialize all of the points \(q_1, \dots, q_{(d-3)/2}\) together to a common general point \(q\in C\).  Since the property of being perfectly balanced is open (for example by the semicontinuity theorem, since it is equivalent to cohomological vanishing) it suffices to show that 
  \begin{equation}\label{eq:mod_NX}
  N_{X}[2p_1 \posmod N_{X \to q}] \cdots [2p_{(d-3)/2} \posmod N_{X \to q}] \simeq N_{X}[2p_1 + \cdots + 2p_{(d-3)/2} \posmod N_{X \to q}]
  \end{equation}
  is perfectly balanced.  
 
 
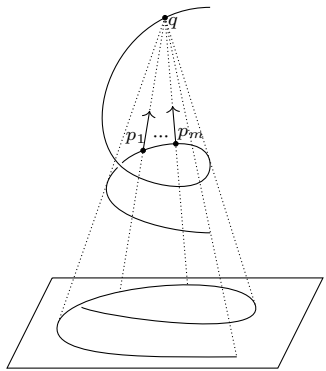
\begin{center}
\begin{tikzpicture}[scale=.6]
\draw[densely dotted] (1, 4.77) -- (3, -1.6);
\draw[densely dotted] (1, 4.77) -- (-1.37, -2);
\draw[densely dotted] (1, 4.77) -- (2.59, -2.75);
\draw[densely dotted] (1, 4.77) -- (0, -1.31);
\draw[densely dotted] (1, 4.77) -- (1.5, -1.155);
\draw[ultra thick, white] (1.24, 1.975) -- (1.168, 2.8135);
\draw[ultra thick, white] (0.515, 1.825) -- (0.6605, 2.7085);
\draw (2, 5) .. controls (0, 5) and (-1, 2.5) .. (0, 1.5);
\draw (0, 1.5) .. controls (0.5, 1) and (2, 0.75) .. (2, 1.5);
\draw (0.05, 1.55) .. controls (0.5, 2) and (2, 2.25) .. (2, 1.5);
\draw (-0.05, 1.45) .. controls (-1, 0.5) and (1, 0) .. (2, 0);
\filldraw (1, 4.77) circle[radius=0.05];
\draw (-1.5, -1) -- (4.5, -1) -- (3.5, -3) -- (-2.5, -3) -- (-1.5, -1);
\draw (3, -1.6) .. controls (2.75, -0.80375) and (-1.07375, -1.15375) .. (-1.37, -2);
\draw (-1.37, -2) .. controls (-1.66625, -2.84625) and (1, -2.75) .. (2.59, -2.75);
\draw (3, -1.6) .. controls (3.25, -2.39625) and (-0.45, -1.85) .. (-0.85, -1.65);
\filldraw (0.515, 1.825) circle[radius=0.05];
\filldraw (1.24, 1.975) circle[radius=0.05];
\draw[->] (1.24, 1.975) -- (1.168, 2.8135);
\draw[->] (0.515, 1.825) -- (0.6605, 2.7085);
\draw (0.36, 2.1) node{{\tiny \(p_1\)}};
\draw (1.59, 2.2) node{{\tiny \(p_m\)}};
\draw (0.91, 2.1) node{{\tiny \(\cdots\)}};
\draw (1.18, 4.66) node{{\tiny \(q\)}};
\end{tikzpicture}
\end{center}
 
 Consider the projection exact sequence~\eqref{eq:proj} for projection from \(q\).  By~\eqref{eq:mod_ses}, the modified normal bundle in \eqref{eq:mod_NX} sits in the induced short exact sequence
 \begin{equation}\label{eq:proj_d}
 0 \to N_{X \to q}(2p_1 + \cdots + 2p_{(d-3)/2}) \to  N_{X}[2p_1 + \cdots + 2p_{(d-3)/2} \posmod N_{X \to q}] \to N_{\pi_q}(q) \to 0.
 \end{equation}
 By \eqref{eq:pointing}, the subbundle is isomorphic to \(\O_{\pp^1}( (d+3)/2 + 2 + (d-3)) \simeq \O_{\pp^1}((3d+1)/2)\).
 As observed earlier, adjunction implies that the quotient is isomorphic to \(\O_{\pp^1}(3(d+3)/2 - 4) \simeq \O_{\pp^1}((3d + 1)/2)\).  Thus the sequence \eqref{eq:proj_d} is perfectly balanced, and so \eqref{eq:mod_NX} is perfectly balanced as well.   We conclude by Lemma~\ref{lem:1sec} that the same must be true for a general rational curve of odd degree \(d\).
 \end{proof}

 
 \newcommand{\etalchar}[1]{$^{#1}$}

\end{document}